\documentclass[10 pt,reqno]{amsart}
\usepackage{color}
\usepackage{enumerate}
\usepackage{amssymb,amsmath,amsthm,latexsym,mathrsfs}
\usepackage{amsbsy,amsfonts,mathtools,slashed}
\usepackage{graphicx,color,yfonts}
\usepackage[numbers,sort]{natbib}
\usepackage[latin9]{inputenc}

\usepackage[colorlinks=true]{hyperref}
\hypersetup{linkcolor=red,citecolor=blue,filecolor=dullmagenta,urlcolor=blue} 

\newtheorem{thm}{Theorem}[section]
\newtheorem{lemma}[thm]{Lemma}
\newtheorem{prop}[thm]{Proposition}

\newtheorem{rem}[thm]{Remark}

\makeatother
\numberwithin{equation}{section}
\numberwithin{figure}{section}
\theoremstyle{plain}
\theoremstyle{plain}
\theoremstyle{plain}

\makeatother
\numberwithin{equation}{section}
\numberwithin{figure}{section}
\theoremstyle{plain}
\theoremstyle{plain}
\theoremstyle{plain}

\newcommand{\les}{\lesssim}

\newcommand{\gam}{{\gamma}}

\newcommand{\vp}{{\varphi}}
\newcommand{\ve}{{\varepsilon}}
\newcommand{\de}{{\delta}}

\newcommand{\al}{{\alpha}}
\newcommand{\ka}{{\kappa}}

\newcommand{\R}{{\mathbb R}}
\newcommand{\Z}{{\mathbb Z}}

\newcommand{\C}{{\mathbb C}}
\newcommand{\N}{{\mathbb N}}

\newcommand{\p}{\partial}

\newcommand{\brad}{{\bra{D}}}
\newcommand{\braxi}{{\bra{\xi}}}
\newcommand{\brat}{{\bra{t}}}

\newcommand{\thez}{{\theta_0}}
\newcommand{\theo}{{\theta_1}}
\newcommand{\thet}{{\theta_2}}

\newcommand{\thej}{{\theta_j}}

\newcommand{\kaz}{{\ka_0}}
\newcommand{\kao}{{\ka_1}}
\newcommand{\kat}{{\ka_2}}

\newcommand{\doth}{{\dot{H}}}

\newcommand{\cm}{{\rm CM}}
\newcommand{\nmax}{{N_{012}^{\max}}}
\newcommand{\nmin}{{N_{012}^{\min}}}

\def\normo#1{\left\|#1\right\|}

\def\bra#1{\left\langle #1\right\rangle}

\def\wt#1{\widetilde{#1}}
\def\wh#1{\widehat{#1}}

\begin{document}
\title[Maxwell-Dirac system in $\mathbb R^{1 + 4}$]{Scattering results for the (1+4) dimensional massive Maxwell-Dirac system under Lorenz gauge condition}

\author{Kiyeon Lee}
\address{Stochastic Analysis and Application Research Center(SAARC), Korea Advanced Institute of Science and Technology, 291 Daehak-ro, Yuseong-gu, Daejeon, 34141, Republic of	Korea}
\email{kiyeonlee@kaist.ac.kr}


\begin{abstract}
	This paper investigates the \emph{massive} Maxwell-Dirac system under the Lorenz gauge condition in (4+1) dimensional Minkowski space. The focus is on establishing global existence and scattering results for small solutions on the weighted Sobolev class. The imposition of the Lorenz gauge condition transforms the Maxwell-Dirac system into a set of Dirac equations coupled with an electromagnetic potential derived from five quadratic wave equations. To achieve a comprehensive understanding of the global solution and its behavior, we employ various energy estimates based on the space-time resonance argument. This involves addressing diverse resonance functions arising from the free Dirac and wave propagators. Additionally, we identify the space-time resonant sets associated with the \emph{massive} Maxwell-Dirac system.
\end{abstract}

\maketitle

\tableofcontents

\section{Introduction}

The subject of this paper is the global existence of the (1+4) dimensional \emph{massive} Maxwell-Dirac system and a linear scattering phenomenon of the small solution under the Lorenz gauge condition. To handle the nonlinearity that possesses both of Klein-Gordon type phase and wave type phase, we utilize the space-time resonance argument. To this end, we investigate the time, space, and space-time resonant sets, combined with the Dirac and Maxwell parts. Then we proceed with a bootstrap argument, based on the energy estimates.  See Sections \ref{sec:ideas} and \ref{sec:resonance} for the details. Compared to \emph{massless} Maxwell-Dirac system, the global solutions exhibit the \emph{linear} behavior  in our main result.

\subsection{Maxwell-Dirac system}
In this paper we consider Cauchy problem of (1+4)-dimensional Maxwell-Dirac system:
\begin{align}
	\left\{ \begin{aligned}
		i\al^\mu \textbf{D}_{\mu} \psi &= m \gamma^0 \psi,\\
		\partial^\nu F_{\mu\nu} &= - \bra{\psi, \al^\mu \psi}.
	\end{aligned}
	\right.\label{md} \tag{MD}
\end{align}
where  the spinor field  is $\psi : \R^{1+4} \to \C^{4}$, the gauge fields are $A_\mu : \R^{1+4} \to \R$  and the covariant derivative is  $\textbf{D}_\mu = \partial_\mu -iA_\mu$ for $\mu = 0, \cdots , 4$.  The curvature is defined by $F_{\mu\nu} = \partial_\mu A_\nu - \partial_\nu A_\mu$. The Dirac matrices are Hermitian and have the relation
\begin{align*}
	\gamma^\mu \gamma^\nu + \gamma^\nu \gamma^\mu = 2\delta^{\mu\nu}I_4,
\end{align*}
where  $I_4$ is a $4\times 4$ identity matrix. We also define $\al^\mu = \gamma^0 \gamma^\mu$ for $\mu =0, \, \cdots, \,4$. The $\langle \cdot , \cdot \rangle$ denotes a standard complex inner product. Greek indices indicate the space-time components $\mu,\nu = 0,\, 1,\,2,\,3,\,4$ and roman indices mean the spatial components $j=1,\, 2,\,3,\,4$ in the sequel.  The Einstein summation convention is in effect with Greek indices summed over $\mu=0,\cdots, 4$ and Latin indices summed over the spatial variables $j = 1, \cdots, 4$. Thus $\p^\mu = \eta^{\mu\nu}\p_\nu$, and $\p_0=\p_t$ with Minkowski metric $\eta = {\rm diag}(-1,1,\cdots,1)$. We call \emph{massive} and \emph{massless} \eqref{md} if the mass parameter is $m > 0$ and $m=0$, respectively.

Maxwell-Dirac system is the Euler-Lagrange equations for $\textbf{S}[A_\mu,\psi]$, where
\begin{align*}
	\textbf{S}[A_\mu,\psi] = \iint_{\R^{1+4}} -\frac14 F^{\mu\nu}F_{\mu\nu} + i \bra{\gamma^\mu \textbf{D}_\mu \psi, \gamma^0 \psi} - m \bra{\psi,\psi} dx dt.
\end{align*}
Computing this Euler-Lagrangian, we arrive at the \eqref{md}. The equations \eqref{md} model an electron in electromagnetic field and form a fundamental system in quantum electrodynamics. For detailed description, we refer to \cite{book:bjorken-drell,book:schwartz}.

The one of basic features of \eqref{md} is the gauge invariance. Indeed, \eqref{md} is invariant under the gauge transformation $
	(\psi,A) \longmapsto (e^{i\chi}\psi, A-d\chi),$ for a real-valued function $\chi$ on $\R\times\R^{4}$. For the sake of concreteness of our discussion, let us choose Lorenz gauge, which is defined by the condition  
\begin{align}\label{gauge:lorenz}
	\partial^\mu A_\mu =0.
\end{align}
Under this gauge condition, \eqref{md} becomes 
\begin{align}
\left\{ \begin{aligned}(-i\partial_t +  \alpha \cdot D  + m\gam^0)\psi &=  A_\mu\alpha^\mu \psi\quad\mathrm{in}\;\;\mathbb{R}^{1+4},\\
\square A_{\mu} & = -\bra{\psi , \alpha_\mu \psi },
\end{aligned}
\right.\label{maineq:md-lorenz}
\end{align}
with initial data
\begin{align}\label{eq:initial}
	\psi(0) = \psi_0,\quad A_\mu(0)& = a_{ \mu}, \quad\partial_tA_\mu(0) = \dot{a}_{ \mu}.
\end{align}
We denoted  $\al_\mu = \al^\mu$ in the second equation of \eqref{maineq:md-lorenz}. The momentum operator $D = -i\nabla$ is defined by $D_j = -i\partial_j\;(j = 1, \cdots,4)$ and $\al = (\al^1,\cdots,\al^4)$.  The positive constants $m$ denotes the mass of Fermion. We also denote $\square=-\partial_{t}^{2}+\Delta$. For other common choices, there are Coulomb gauge condition $\p^j A_j =0$ and temporal gauge condition $A_0=0$.

\subsection{Previous works} There is a large amount of literature dealing with the problem of local and global well-posedness and asymptotic behavior of small solutions of \eqref{md}. For early work in \cite{gross1966,bour1996}, they considered the local well-posedness of \eqref{md} on $\R^{1+3}$ and Georgiev \cite{geor1991} proved the global existence for small, smooth initial data. Later, D'Ancona-Foschi-Selberg \cite{anfosel2010} developed the results up to almost optimal regularity $(\psi_0, a_\mu, \dot{a}_\mu) \in H^\ve \times H^{\ve+\frac12} \times H^{\ve-\frac12}$ on $\R^{1+3}$ in Lorenz gauge condition. They exploited the spinorial null structures, which stem from Dirac projection operators.  D'Ancona-Selberg \cite{ansel} extended their approach to \eqref{md} on $\R^{1+2}$ and proved global well-posedness in the charge class $L^2 \times H^\frac12 \times H^{-\frac12}$. Regarding Local well-posedness results for \eqref{md} with Coulomb gauge condition, we refer to \cite{bemaupoup1998,masnaka2003-imrn}. Moreover, Masmoudi and Nakanishi \cite{masnaka2003-cmp} showed the unconditional uniqueness results for (1+3) dimensional \eqref{md} for Coulomb gauge condition.

Concerning an asymptotic behavior of global solution to \eqref{md} in physical space $\R^{1+3}$, we refer to \cite{psa2005,flasimon1997}.   Psarelli \cite{psa2005} showed the global existence and asymptotic behavior for  (3+1) Minkowski space the \emph{massive} \eqref{md}  with compactly supported small initial data.   Regarding the other asymptotic behavior results of the global solution to \eqref{md}, the modified scattering results are shown by some results. Gavrus and Oh \cite{gaoh} established the global well-posedness of \emph{massless} \eqref{md} with Coulomb gauge condition for (1+d) Minkowski space $d \ge 4$ and modified scattering of the solution. They also showed the global well-posedness in $\doth^\frac{d-3}2\times \doth^{\frac{d-2}2} \times \doth^{\frac{d-4}2}$, which is absolutely critical regularity.  Regarding Dirac equations with the magnetic potential, D'Ancona and Okamoto  \cite{ancooka2017} showed an scattering results. Vanishing the magnetic field, in \cite{CKLY2022,cloos}, they established the modified scattering results for 3 dimensional \eqref{md} under Lorenz gauge condition,  independently. 

The Maxwell-Klein-Gordon system (MKG), a scalar counterpart of \eqref{md}, has been studied  by many authors. Regarding the global well-posedness for the (1+4) dimensional \emph{massless} (MKG), we refer the reader to \cite{kristertata2015-duke,rodtao2004}. 
In \cite{krieluhr2015,ohtataru2016-inven,ohtata2018}, they independently studied the scattering results for  \emph{massless} (MKG)  in (1+4) Minkowski dimension under the Coulomb gauge condition. Concerning the global existence for (1+3) Minkowski dimension, we refer to \cite{seltes2010}. In particular, in \cite{yang2018,yangyu2019,cankaulin2019}, the authors presented an asymptotic behavior of the solution to the \emph{massless} (MKG) under Coulomb and Lorenz gauge condition in (1+3) dimension. While \emph{massless} (MKG) are concerned by many authors, there is a few scattering results for \emph{massive} (MKG). We mention the recent result in \cite{gav2019} which establishes global regularity for the modified scattering results of the \emph{massive} (MKG) under the Coulomb gauge condition. See also \cite{klaiwangyang2020,fangwangyang2021} and \cite{dowya2024,dolimayu2024} for the asymptotic behavior of the (1+3) dimensional \emph{massive} (MKG) and wave-Klein-Gordon type systems, respectively.

\subsection{Main theorem}
As seen from the listed results above, while some modified scattering results are known,  the \emph{linear} scattering results of both \emph{massive} and \emph{massless} \eqref{md} are unknown in any dimension even if any gauge conditions are imposed. In this paper, we consider the global existence of the solution to \eqref{maineq:md-lorenz} and scattering results for \emph{massive} \eqref{maineq:md-lorenz} ($m>0$) on Minkowski space $\R^{1+4}$. For the sake of simplicity, we normalize $m=1$. Our main result is the following:
\begin{thm}\label{mainthm}
Let $n$ is sufficiently large number. Then we have the following.
	\begin{itemize}
		\item[(1)] Let initial data \eqref{eq:initial} satisfy that, for some $\ve_0>0$,
	\begin{align}\begin{aligned}
			\|\psi_{0}\|_{H^n}+\| \bra{x}^2\psi_{0}\|_{L^2}& \\
			\|(a_{\mu}, \dot a_{ \mu})\|_{H^n \times H^{n-1}} + \sum_{k=1,2}\|(x^ka_{\mu}, x^k\dot a_{\mu})\|_{\doth^{k} \times \doth^{k-1}} &
	\end{aligned}< \ve_{0}.\label{condition-initial}
\end{align}
	Then, there exists a global solution $(\psi(t), A_\mu(t))$ to \eqref{md} under the Lorenz gauge condition \eqref{gauge:lorenz}.
	
	\item[(2)] Under the assumption \eqref{condition-initial}, the solution decays as follows:
	\begin{align}\label{thm:decay}
		\|\psi(t)\|_{L^\infty} \les \brat^{-2+\beta +\de}\ve_0, \;\; \mbox{ and }\;\; \|A_{\mu}(t)\|_{L^\infty} \les \brat^{-\frac32}\ve_0,
	\end{align}
	for some small $\beta,\de =\beta(n),\de(n)>0$.
	\item[(3)] Moreover, the solution $(\psi(t), A_\mu(t) )$ to \eqref{md} scatters linearly in $L^2 \times \doth^{1} \times L^2$ as $t \to \infty$. Indeed, there exist $(\psi^\infty(t), A_\mu^\infty(t))$ such that
	\begin{align*}
		\normo{\psi(t) - \psi^\infty(t)}_{L^2} +\normo{(A_\mu(t), \p_t A_\mu(t))- (A_\mu^\infty(t),\p_t A_\mu^\infty(t))}_{\doth^1\times L^2} \xrightarrow{t\to \infty} 0,
	\end{align*}
	and $(\psi^\infty(t), A_\mu^\infty(t) )$ is a solution to the linear system \eqref{md}:
	\begin{align*}
		\left\{ \begin{aligned}i\al^\mu \p_\mu \psi^\infty &= m\gamma^0 \psi^\infty,\\
			\square A_{\mu}^\infty & = 0.
		\end{aligned}
		\right.
	\end{align*}
\end{itemize}
\end{thm}

In view of \eqref{thm:decay}, we have pointwise decay of the Maxwell component at the optimal rate of $t^{-\frac32}$, whereas the Dirac component decays almost optimal at a rate of $t^{-2+}$. See Remark \ref{rem:optimal-timedecay} and Propositions \ref{prop:timedecay} and \ref{prop:timedecay-maxwell}. Concerning scattering results, the function space can be replaced with $H^n \times \doth^n \times \doth^{n-1}$, instead of $L^2 \times \doth^1 \times L^2$. See Remark \ref{rem:scattering}.

\begin{rem}
	There are various candidates of choice of the gauge condition. If we impose Coulomb gauge condition $\p^j A_j =0$,
	 Maxwell-Dirac system \eqref{md} is rewritten by
	\begin{align}
		\left\{ \begin{aligned}(-i\partial_t +  \alpha \cdot D  + m\gam^0)\psi &=  A_\mu\alpha^\mu \psi\quad\mathrm{in}\;\;\mathbb{R}^{1+4},\\
			\Delta A_{0} & = -|\psi|^2,\\
			\square A_{\mu} & = -\bra{\psi , \alpha_\mu \psi }.
		\end{aligned}
		\right.\label{maineq:md-coulomb}
	\end{align}
Then the smallness assumption  may guarantee the smallness of $A_0$. The main argument of this paper can be applied for $A_j$'s in \eqref{maineq:md-coulomb} by setting the same a priori assumption to \eqref{eq:assumption-apriori}. Therefore, we may anticipate that the global existence and linear scattering results for \eqref{maineq:md-coulomb} will be proven.
\end{rem}

\begin{rem}
	In \eqref{thm:decay}, the parameters $\beta$ and $\de$ are depending on regularity $n$. sufficiently large number $n$ implies the almost optimal rate of Dirac operator $t^{-2+}$. See Remark \ref{rem:parameter} for more precise conditions of $\beta$ and $\de$.
\end{rem}

Thanks to the time decay effect, which \emph{massive} Dirac operator give, Theorem \ref{mainthm} shows the linear scattering results for the \emph{massive} \eqref{md} even if we impose Coulomb gauge condition, whereas  the \emph{massless} \eqref{md} exhibited the modified scattering
in \cite{gaoh}. We note that our main results is the first result establishing the linear scattering phenomenon for \eqref{md}.

\subsection{Main idea}\label{sec:ideas} 
The proof of the global solution and scattering results is based on a bootstrap argument under the suitable a priori assumption with the weighted energy estimates. To this end, we use the space-time resonant argument, developed by Germain-Masmoudi-Shatah \cite{gemasha2008,gemasha2012-annals,gemasha2012-jmpa}.By employing Duhamel's formula, the nonlinear terms of \eqref{maineq:md-lorenz} give rise to several oscillatory integrals, with the stationary points of the phase function dictating their large-time behaviors.

Our analysis begins with addressing the space, time, and space-time resonance of \textit{phases} originating from the Dirac spinor and wave operator within the oscillatory integrals, arising from Dirac and wave propagators. Specifically, through a standard reformulation with the Dirac projection operator and first-order wave equations, we obtain the system of the following nonlinearities $\mathbf{f}_{\Theta}$ and $\mathbf{F}_{\mu,K}$ for the Dirac and Maxwell parts, respectively:
\begin{align*}
	\mathbf f_{\,\Theta}(t, \xi) &:= \int_0^t\!\!\int_{\R^4} e^{is\,p_\Theta(\xi, \eta)  }\Pi_{\theta_0}(\xi)\widehat{ F_{\mu, \theta_2}}(s, \eta) \alpha^\mu \widehat{ f_{\theta_1}}(s, \xi-\eta)\,d\eta ds,\\
	\mathbf F_{\mu,K}(t, \xi) &:= |\xi|^{-1}\int_{0}^{t}\!\!\int_{\R^4} e^{is q_K(\xi,\eta)} \bra{\widehat{f_\kat}(s,\eta),\alpha_\mu\wh{f_\kao}(s,\xi+\eta)}\, d \eta ds,\\
	p_\Theta(\xi, \eta) &:= \theta_0 \bra{\xi} - \theta_1\bra{\xi-\eta} + \theta_2|\eta|,\\
	q_K(\xi, \eta) &:= -\ka_0 |\xi| - \ka_1 \bra{\xi+\eta}  + \ka_2\bra{\eta},
\end{align*}
for 3-tuples $\Theta = (\thez, \theo, \thet)$ and $K=(\kaz,\kao,\kat)$ where $\thej, \ka_j \in \{+,-\}\,(j=0,1,2)$  (see Section \ref{sec:reformul} for the rigorous derivation). The functions $p_\Theta$ and $q_K$ are commonly referred to as \emph{phases}. These phases exhibit various sign relations, requiring consideration of different resonance cases. While the phases introduce space-time resonance, the nonlinearities \eqref{md} lack sufficient null structures to eliminate these resonances. As observed in \cite{anfosel2010,huhoh2016,gaoh}, the interaction between projection operators and Dirac matrices $\al_j$ induces sign-changed projection parts and Riesz transform parts:
\begin{align*}
	\al^j \Pi_{\theta}(\xi) =  \Pi_{-\theta}(\xi)\al^j + \theta \frac{\xi_j}{|\xi|}.
\end{align*}
 Depending on the sign relation, cases that require null structures emerge, while the interaction generates terms without null structures, regardless of the sign relation.

To overcome this lack of null structure, we impose a homogeneous Sobolev norm in the a priori assumption of the Maxwell part. This assumption provides as a null structure in the estimates for the Maxwell part. However, it also leads to singularities in the estimates of the Dirac part. To handle these singularities, we exploit the almost optimal time decay effect of the Dirac spinor and the dimensional advantage.

As we have observed in numerous previous works, the relationship of \eqref{md} to (MKG) becomes apparent. The emergence of the Dirac equation can be traced back to an endeavor to find the square root of the Klein-Gordon equation, to formulate an equation that is first-order in time. Consequently, squaring the Dirac component of the system yields an equation that bears a resemblance to the Klein-Gordon segment of (MKG). Nevertheless, as pointed out in \cite{anfosel2010}, this idea appears to have limited applicability, as squaring the Dirac equation results in the loss of much of its spinorial null structure. It is worth noting, however, that our proof in this paper does not depend on exploiting this null structure. Consequently, our approach can be extended to encompass the \emph{massive} (MKG).

\begin{rem}
	It remains still open to obtain the global existence and asymptotic behavior of \eqref{md} in (1+3) dimension. The extension of these results to three dimensions is quite limited. Since the lower dimension implies the lower time decay effect, $L^2$--norm of nonlinearities are not integrable in time on $\R^{1+3}$. In view of the modified scattering results \cite{ancooka2017,CKLY2022,cloos}, then we anticipate that the global solution to \eqref{md} in $\R^{1+3}$ exhibits a long range behavior which requires suitable phase modification. 
\end{rem}

\medskip

\noindent\textbf{Outlines of this paper.} In Section \ref{sec:pre}, we introduce notations and reformulate \eqref{md} with Lorenz gauge condition \eqref{gauge:lorenz}. We also identify space, time, and space-time resonant sets for various sign relations and assume a priori assumption for our main proof. In the end of this section, we give the time decay estimates for Dirac and Maxwell equations under the a priori assumption. In Section \ref{sec:useful}, we introduce some useful lemmas, used throughout the paper. Then we prove our main theorem  in Section \ref{sec:mainproof}. We show the existence of a global solution to \eqref{md} under Lorenz gauge condition by the bootstrap argument, as well as scattering of the global solution.  In Sections \ref{sec:1st-dirac}--\ref{sec:2nd-esti-maxwell}, we close the a priori  assumption by using the space-time resonance argument.

\section{Preliminaries}\label{sec:pre}

\noindent{\bf Notations.}
\ \\
$(1)$ $\bra{\cdot}$ denotes $(1+ |\cdot|^2)^{\frac12}$. We also denote $\bra{D} := \mathcal F^{-1}(\bra{\xi})$ and $|D| := \mathcal F^{-1}(|\xi|)$.\\
\ \\
$(2)$ (Mixed-normed spaces) For a Banach space $X$ and an interval $I$, $u \in L_I^q X$ iff $u(t) \in X$ for a.e. $t \in I$ and $\|u\|_{L_I^qX} := \|\|u(t)\|_X\|_{L_I^q} < \infty$. Especially, we abbreviate $L^p=L_x^p$ for the spatial norm and indicate the subscripts for only Fourier space norm. \\

\noindent$(3)$ (Sobolev spaces)  Let $s \in \R$. We define homogeneous Sobolev spaces $\|u\|_{\doth^s} := \||D|^s u\|_{L^2}$.  For inhomogeneous spaces, we define $\|u\|_{H^s} := \|\bra{D}^s u\|_{L^2}$. Moreover, $\|u\|_{W^{s,p}} := \|\bra{D}^s u\|_{L^p}$ for $p\neq 2$.\\

\noindent$(4)$ Different positive constants depending only on $n$, $\ve_0$ are denoted by the same letter $C$, if not specified. $A \lesssim B$ and 
$A\gtrsim B $ mean that $A \le CB$ and $A \geq  C^{-1}B$  respectively for some $C>0$. $A \sim B$  means  that $A \lesssim B$ and $A \gtrsim B$.\\

\noindent$(5)$ (Littlewood-Paley operators) Let $\rho$ be a
Littlewood-Paley function such that $\rho\in C_{0}^{\infty}(B(0,2))$
with $\rho(\xi)=1$ for $|\xi|\le1$ and define $\rho_{N}(\xi):=\rho\left(\frac{\xi}{N}\right)-\rho\left(\frac{2\xi}{N}\right)$
for $N\in2^{\mathbb{Z}}$. Then we define the frequency projection
$P_{N}$ by $\mathcal{F}(P_{N}f)(\xi)=\rho_{N}(\xi)\widehat{f}(\xi)$,
and also $\rho_{\le N_{0}}:=1-\sum_{N>N_{0}}\rho_{N}$. For $N\in2^{\mathbb{Z}}$
we denote $\widetilde{\rho_{N}}=\rho_{N/2}+\rho_{N}+\rho_{2N}$. In
particular, $\widetilde{P_{N}}P_{N}=P_{N}\widetilde{P_{N}}=P_{N}$
where $\widetilde{P_{N}}=\mathcal{F}^{-1}\widetilde{\rho_{N}}\mathcal{F}$.
Especially, we denote $P_{N}f$ by $f_{N}$ for any measurable function
$f$.\\

\noindent$(6)$ For the dyadic numbers $ N_j, N_k, N_\ell \in 2^{\Z}$, we denote $\max( N_j, N_k, N_\ell)$ and $\min( N_j, N_k, N_\ell)$ by $N_{jk\ell}^{\max}$ and $N_{jk\ell}^{\min}$, respectively.\\

\noindent$(7)$ Let $\textbf{A}=(A_i), \textbf{B}=(B_i) \in \R^n$. Then $\textbf{A} \otimes \textbf{B}$ denotes the usual tensor product such that $(\textbf{A} \otimes \textbf{B})_{ij} = A_iB_j$. We also denote a tensor product of $\textbf{A} \in \C^n$ and $\textbf{B} \in \C^m$  by a matrix $\textbf{A} \otimes \textbf{B} = (A_iB_j)_{\substack{i=1,\cdots,n\\i=1,\cdots,m}}$. For simplicity, we use the simplified notation
$$
[\mathbf A]^k = \overbrace{\mathbf A \otimes \cdots \otimes \mathbf A}^{k\; \text{times}},\qquad \nabla^k = \overbrace{\nabla \otimes \cdots \otimes \nabla}^{k\; \text{times}}.
$$
The product of  $\mathbf A$ and $f \in \C^n$ is given by $\mathbf A f = \mathbf A \otimes f$.

\subsection{Set up for Maxwell-Dirac system under Lorenz gauge condition}\label{sec:reformul} We decompose the Dirac spinor field $\psi$ into half waves, i.e., $\psi_+$ and $\psi_-$.
To do this we define projection operators $\Pi_{\theta}(D)$ for $\theta \in \{+, -\}$
by
\[
\Pi_{\theta}(D) :=\frac{1}{2}\left( I_4 + \theta \frac{\alpha \cdot  D + \gam^0}{\bra{D}}\right).
\]
We denote the symbol of $\Pi_\theta(D)$ by $\Pi_\theta(\xi)$.  Then we have
\[
\Pi_\theta(D) + \Pi_{-\theta}(D) = I_4, \quad \Pi_\theta(D) \Pi_{-\theta}(D) = 0, \quad \mbox{ and }\quad [\Pi_\theta(D)]^2 = \Pi_\theta(D).
\]
Let $\psi_\theta = \Pi_\theta(D) \psi$ for $\theta \in \{+,-\}$. Then we obtain the following decoupled equations from the Dirac part of \eqref{maineq:md-lorenz}:
\begin{align}\label{eq:half-dirac}
\left\{
\begin{array}{l}
(-i\partial_t + \theta\bra{D})\psi_\theta = \Pi_\theta(D)\big(A_\mu \alpha^\mu \psi\big),\\
\psi_\theta(0):= \psi_{0, \theta}.
\end{array}\right.
\end{align}
 Note that $\psi = \psi_+ + \psi_-$.

We now decompose gauge field as $A_\mu = A_{\mu,+} + A_{\mu,-}$ with
$$
A_{\mu, \ka} = \frac12\Big(1 + \ka |D|^{-1}(-i\partial_t)\Big)A_\mu,
$$
for $\ka \in \{+, -\}$. Then the Maxwell part of \eqref{maineq:md-lorenz} is rewritten as the following equations:
\begin{align}\label{eq:half-maxwell}
\left\{\begin{array}{l}
(i\partial_t + \ka|D|)A_{\mu, \ka} = \ka\frac{1}2|D|^{-1}\bra{\psi, \alpha_\mu \psi},\\
A_{\mu, \ka}(0) := a_{\mu, \ka} := \frac12(a_{ \mu} - \ka i|D|^{-1}\dot a_{ \mu}).
\end{array}\right.
\end{align}
Therefore \eqref{maineq:md-lorenz} becomes that
\begin{align}\label{eq:maineq-half}
	\left\{\begin{aligned}
		(-i\partial_t + \theta\bra{D})\psi_\theta &= \Pi_\theta(D)\big(A_\mu \alpha^\mu \psi\big),\\
		(i\partial_t + \ka|D|)A_{\mu, \ka} &= \ka\frac{1}2|D|^{-1}\bra{\psi, \alpha_\mu \psi},
	\end{aligned}\right.
\end{align}
with initial data $(\psi_\theta(0), A_{\mu,\ka}(0)) = (\psi_{0,\theta}, a_{\mu,\ka})$, for $\theta, \ka \in \{+,-\}$.

Now by Duhamel's principle, \eqref{eq:maineq-half} can be converted into
\begin{align}
\psi_{\theta_0}(t) &= e^{-\theta_0 it\brad} \psi_{0, \theta_0}  +   i \sum_{\theta_1, \theta_2 \in \{\pm \}}\int_0^t e^{-\theta_0 i(t-s)\brad}\Pi_{\theta_0}(D)\big(A_{\mu, \theta_2} \alpha^\mu \psi_{\theta_1}\big)(s)\,ds, \label{eq:duhamel-dirac}\\
A_{\mu, \ka_0}(s) &= e^{\ka_0is|D|}a_{\mu, \ka_0} -\ka_0 \frac{i}2\sum_{\ka_1, \ka_2 \in \{\pm\}}\int_{0}^{s} e^{\ka_0i(t-s)|D|}|D|^{-1}\bra{\psi_{\ka_1},\alpha_\mu\psi_{\ka_2}}(s)\,ds.  \label{eq:duhamel-maxwell}
\end{align}

To keep track of the scattering state we define propagator-adapted fields $f_{\theta}$ and $F_{\mu, \ka}$ by
\begin{align}\label{eq:interaction}
f_\theta(t) = e^{\theta i t \brad} \psi_\theta(t) \;\;\mbox{ and }\;\; F_{\mu, \ka}(t) = e^{-\ka it|D|}A_{\mu, \ka}(t).	
\end{align}
Then by acting propagators on both sides of \eqref{eq:duhamel-dirac} and \eqref{eq:duhamel-maxwell}, respectively, we have
\begin{align}
f_{\theta_0}(t) &= \psi_{0, \theta_0} + i\sum_{\theta_1, \theta_2 \in \{\pm \}}\int_0^t e^{\theta_0 is\brad}\Pi_{\theta_0}(D)\big(A_{\mu, \theta_2} \alpha^\mu \psi_{\theta_1}\big)(s)\,ds, \label{eq:profile-dirac} \\
F_{\mu, \ka_0}(t) &= a_{ \mu, \ka_0} - \ka_0 \frac{i}2\sum_{\ka_1, \ka_2 \in \{\pm \}} \int_{0}^{t} e^{-\ka_0is|D|}|D|^{-1}\bra{\psi_{\ka_1},\alpha_\mu\psi_{\ka_2}}(s)\,ds,\label{eq:profile-maxwell}
\end{align}
for $\thez, \kaz \in \{+,-\}$. By taking Fourier transform, one gets the frequency representation as follows:
\begin{align*}
\widehat f_{\theta_0}(t, \xi) &=  \widehat {\psi_{0, \theta_0}} + i\sum_{\theta_1, \theta_2 \in \{ \pm \}}\mathbf f_{\,\Theta}(t, \xi), \\
\widehat{ F_{\mu, \ka_0}}(t, \xi) &= \widehat{a_{ \mu, \ka_0}}(\xi) - \ka_0\frac{i}2 \sum_{\ka_1, \ka_2 \in \{ \pm \}} \mathbf F_{\,\mu, K}(t, \xi),
\end{align*}
where
\begin{align}
\mathbf f_{\,\Theta}(t, \xi) &:= \int_0^t\!\!\int_{\R^4} e^{is\,p_\Theta(\xi, \eta)  }\Pi_{\theta_0}(\xi)\widehat{ F_{\mu, \theta_2}}(s, \eta) \alpha^\mu \widehat{ f_{\theta_1}}(s, \xi-\eta)\,d\eta ds,\label{eq:interaction-dirac}\\
\mathbf F_{\,\mu,K}(t, \xi) &:= |\xi|^{-1}\int_{0}^{t}\!\!\int_{\R^4} e^{is q_K(\xi,\eta)} \bra{\widehat{f_\kat}(s,\eta),\alpha_\mu\wh{f_\kao}(s,\xi+\eta)}\, d \eta ds,\label{eq:interaction-maxwell}\\
p_\Theta(\xi, \eta) &:= \theta_0 \bra{\xi} - \theta_1\bra{\xi-\eta} + \theta_2|\eta|,\label{eq:phase-dirac}\\
q_K(\xi, \eta) &:= -\ka_0 |\xi| - \ka_1 \bra{\xi+\eta}  + \ka_2\bra{\eta},\label{eq:phase-maxwell}
\end{align}
for 3-tuples $\Theta = (\thez, \theo, \thet)$ and $K=(\kaz,\kao,\kat) $. We call $p_\Theta, q_K$ the \textit{phase}.

\subsection{Space-time resonance for Maxwell-Dirac system}\label{sec:resonance} We proceed our proof by a bootstrap argument based on space-time resonance argument 
\cite{gemasha2008,gemasha2012-jmpa,gemasha2012-annals} and \cite{gunakatsa2009}, independently, which is our frame work. In this section we investigate a time, space, and space-time resonances for \emph{massive} Maxwell-Dirac system, respectively. For the time resonant set of wave and Klein-Gordon type systems, we refer to \cite{iopau2019,book:iopau2022}. We also refer to \cite{pusa,canher2018-analpde} for the observation of Dirac operator. According to the definition in \cite{gemasha2008}, we define time, space, and space-time resonant sets as follows:

\begin{align*}
	\begin{aligned}
		\mathcal T_{p_\Theta} &:= \left\{ (\xi,\eta) : p_\Theta(\xi,\eta)=0\right\},\\
		\mathcal S_{p_\Theta} &:= \left\{ (\xi,\eta) : \nabla_\eta p_\Theta(\xi,\eta)=0\right\},\\
		\mathcal R_{p_\Theta} &:= \mathcal T_{p_\Theta} \cap \mathcal S_{p_\Theta},
	\end{aligned}\hspace{2cm} \begin{aligned}
	\mathcal T_{q_K} &:= \left\{ (\xi,\eta) : q_K(\xi,\eta)=0\right\},\\
	\mathcal S_{q_K} &:= \left\{ (\xi,\eta) : \nabla_\eta q_K(\xi,\eta)=0\right\},\\
	\mathcal R_{q_K} &:= \mathcal T_{q_K} \cap \mathcal S_{q_K}.
	\end{aligned}
\end{align*}
It is difficult to examine precisely elements of these sets. In view of these resonant sets, it suffices to observe the lower bounds of $p_\Theta, \nabla_\eta p_\Theta$ and $q_\Theta, \nabla_\eta q_\Theta$, respectively.

Let us observe that $p_{\Theta}(\xi, \eta)$ never vanishes, when $\xi\neq 0$ and $\eta \neq 0$, for any sign $\theta_0,\theo,\thet \in \{+, -\}$. In fact, if for some $\xi, \eta$, $p_{\Theta}(\xi, \eta) = 0$, then
$$
(\theta_0\bra\xi + \theta_2|\eta|)^2 = 1 + |\xi -\eta|^2.
$$
This implies that
$$
\braxi|\eta| = |\xi \cdot \eta|\;\;\mbox{and hence}\;\; \braxi \le |\xi|,
$$
which is a contradiction. Hence we have
\begin{align*}
	|p_\Theta(\xi, \eta)| &=  \left|\theta_0\braxi - \theta_1\bra{\xi-\eta} + \theta_2|\eta|\right| =  \left|\frac{(\theta_0\braxi + \theta_2|\eta|)^2 - \bra{\xi-\eta}^2 }{\theta_0\braxi  + \theta_1\bra{\xi-\eta} + \theta_2|\eta| }\right|\\
	&=  \left|\frac{2\theta_0\theta_2\braxi|\eta| + 2 \xi \cdot \eta }{\theta_0\braxi  + \theta_1\bra{\xi-\eta} +\theta_2|\eta|  }\right|.
\end{align*}
If $\theta_0 = \theta_1$, then
\begin{align}\label{eq:resonance-time}\begin{aligned}
		|p_\Theta(\xi,\eta)| &\ge \frac{2|\eta|(\braxi - |\xi|)}{\braxi + \bra{\xi-\eta} + |\eta|}\gtrsim \frac{|\eta|}{\bra{\xi}(\bra{\xi}+\bra{\xi-\eta}+
			\bra{\eta})}.
\end{aligned}\end{align}
If $\theta_0 \neq \theta_1, \theta_0 \neq \theta_2$, then
\begin{align*}
	|p_\Theta(\xi,\eta)| \ge \frac{2|\eta|(\braxi - |\xi|)}{|\braxi - \bra{\xi-\eta} - |\eta||} \gtrsim (\braxi - |\xi|) \gtrsim \bra{\xi}^{-1}.
\end{align*}
Also if $\theta_0 \neq \theta_1, \theta_0 = \theta_2$, then trivially $|p_\Theta(\xi,\eta)| \ge \bra{\xi} +\bra{\xi-\eta}$. Then we have, for $\thez \neq \theo$,
\begin{align}\label{eq:nonresonance-time}
	|p_\Theta(\xi,\eta)|  \gtrsim \bra{\xi}^{-1}.
\end{align}
Therefore $\mathcal{T}_{p_\Theta} = \{\eta=0\}$ when $\thez =\theo$ and $\mathcal {T}_{p_\Theta} = \emptyset$, otherwise.

Let us observe the space resonance of \eqref{eq:interaction-dirac}. For any sign cases, there is no space resonant as follows:
\begin{align}\label{eq:nonresonance-space}
	\left| \nabla_\eta p_\Theta (\xi,\eta)\right| = \left| \theo\frac{\xi-\eta}{\bra{\xi-\eta}} - \thet \frac{\eta}{|\eta|}\right| \gtrsim 1 - \frac{|\xi-\eta|}{\bra{\xi-\eta}} \gtrsim \bra{\xi-\eta}^{-2}.
\end{align}
Therefore we deduce that  $\mathcal{S}_{p_\Theta} = \emptyset$,  for any sign relations, which implies  $\mathcal{R}_{p_\Theta} = \emptyset$.

By virtue of a similarity between $p_\Theta(\xi,\eta)$ and $q_K(\xi,\eta)$, we readily see that
\begin{align}\label{eq:resonance-time-maxwell}
	|q_K(\xi,\eta)| \gtrsim \begin{cases}
		\frac{|\xi|}{\bra{\eta}(\bra{\xi}+\bra{\xi-\eta} +\bra{\eta})}  & \mbox{ for } \kao = \kat,\\
		\bra{\eta}^{-1}  & \mbox{ otherwise}.
	\end{cases}
\end{align}
and
\begin{align}\label{eq:resonance-space-maxwell}
	|\nabla_\eta q_K(\xi,\eta)| \gtrsim \begin{cases}
		\frac{|\xi|}{\bra{\xi+\eta}^3}  & \mbox{ for } \kao = \kat,\\
		\left|\frac{\xi+\eta}{\bra{\xi+\eta}} + \frac{\eta}{\bra{\eta}} \right|  & \mbox{ otherwise}.
	\end{cases}
\end{align}
Note that space resonance of $q_K$ is different from that of $p_\Theta$. Therefore, we have $\mathcal T_{q_K} = \mathcal S_{q_K} =\mathcal R_{q_K} = \{\xi=0\}$ when $\kao =\kat$. 
On the other hand, when $\kao \neq \kat$, 
$\mathcal T_{q_K} = \emptyset$,  $\mathcal S_{q_K}= \{\xi=2\eta\}$, and $\mathcal R_{q_K}= \emptyset$.

\subsection{A priori assumption} As we discussed in introduction, we set a priori assumption in this section for the bootstrap argument.
For $ \ve_{1}>0$ to be chosen
later, we assume a priori smallness of solutions: for some $0<\beta<\de$,
a given time $T>0$, a priori assumptions of Dirac part and Maxwell part are suggested by
\begin{align}\label{eq:assumption-apriori}
	\begin{aligned}
		\|\psi\|_{\Sigma_{T}^\textbf{D}}+\|A_{\mu}\|_{\Sigma_{T}^\textbf{M}} \les\ve_{1},
	\end{aligned}
\end{align}
where
\begin{align*}
	\begin{aligned}
		\|\psi\|_{\Sigma_{T}^\textbf{D}}&:=\|\psi_{+}\|_{\Sigma_{T,+}^{\textbf{D}}} + \|\psi_{-}\|_{\Sigma_{T,-}^\textbf{D}}\\
		\|\psi_{\theta}\|_{\Sigma_{T,\theta}^\textbf{D}} & :=\sup_{t\in[0,T]}\left[\|\psi_{\theta}(t)\|_{H^n}+ \left\|x f_\theta(t) \right\|_{L^2} +\bra{t}^{-\de}\left\|x^{2}f_\theta(t)\right\|_{L^2}  \right],
	\end{aligned}
\end{align*}
and
\begin{align*}
	\begin{aligned}
		\|A_{\mu}\|_{\Sigma_{T}^\textbf{M}}&:=\|A_{\mu, +}\|_{\Sigma_{T,+}^\textbf{M}} + \|A_{\mu, -}\|_{\Sigma_{T,-}^\textbf{M}}\\
		\|A_{\mu, \kappa}\|_{\Sigma_{T,\kappa}^\textbf{M}} &:= \sup_{t\in[0,T]}\left[\bra{t}^{-\frac\beta2 + \frac\de2}\|A_{\mu,\kappa}(t)\|_{{H}^{n}} + \bra{t}^{-\frac\beta2 + \frac\de2}\left\|x F_{\mu, \kappa}(t)\right\|_{\dot H^1}+ \bra{t}^{-\frac18}\left\|x^{2}F_{\mu, \kappa}(t)\right\|_{\dot H^2}  \right].	
		\end{aligned}
\end{align*}

\begin{rem}\label{rem:parameter}
	As mentioned in Introduction, we do not pursue the optimality of regularity condition $n$ and optimal smallness of $\beta, \de$. In our analysis the following conditions are required:	
	\begin{align}\label{eq:condi-parameter}
	 0< \beta < \de, \;\; \frac{56}{n-2} \le \beta,\;\mbox{ and }\;  \beta+\de \le \frac14, 
	\end{align}
	 According to these conditions, $(n,\beta,\de) = (500, \frac4{35},\frac9{70})$ is an one of the parameter candidates. Moreover, we anticipate that it is possible to make the regularity condition $n$ lower by imposing the regularity on the weighted Sobolev norm assumptions \eqref{eq:assumption-apriori} or dividing frequencies elaborately in energy estimates. For the sake of the simplicity of the energy estimates,  we do not treat that.	  
\end{rem}

\begin{rem}
	For the extension of our method, we refer to the asymptotic behavior results of the dispersive and wave system  \cite{hanipusha2013,pusasha2013,iopau2014}.
\end{rem}

\subsection{Time decay estimates}

\begin{prop}[Time decay of amplitudes]\label{prop:timedecay}
	Let $\theta \in \{+,-\}$. Assume
	that $\psi$ satisfies a priori assumption \eqref{eq:assumption-apriori},
	for given $\ve_{1}$ and $T$.  Then there exists $C$
	satisfying that for $0 \le k \le 20$,
	\begin{align}
		\|\psi_{\theta}(t)\|_{W^{k,\infty}}&\le C\bra{t}^{-2 +\beta+\de}\ve_{1},\label{eq:timedecay-d}
	\end{align}
	for any $t \in [0, T]$, where $\beta,\delta$ are in \eqref{eq:assumption-apriori}.
\end{prop}

\begin{rem}\label{rem:optimal-timedecay}
	In view of the parameter condition \eqref{eq:condi-parameter}, for arbitrary small $\beta,\de>0$, we can choose the regularity $n$. Indeed, the global bound results \eqref{eq:timedecay-d} is almost optimal under the sufficiently high regularity condition.
\end{rem}

\begin{proof} Let us first consider the \eqref{eq:timedecay-d}. We write
\[
\psi_{\theta}(t,x) = \frac1{(2\pi)^4}\int_{\mathbb{R}^{4}}e^{it\phi_{\theta}(\xi)}\widehat{f_{\theta}}(\xi)d\xi,
	\]
	where
	\begin{align}
		\phi_{\theta}(\xi):={-\theta}\bra{\xi}+\xi\cdot\frac{x}{t}.\label{theta-phase}
	\end{align}
If $t \les 1$, then by H\"older's inequality we have
\begin{align*}
\left|\bra{D}^k\psi_{\theta}(t,x)\right| &\les  \int_{\R^4}\bra{\xi}^k |\widehat{f_{\theta}}(\xi)|d\xi \les  \left\|\bra{\xi}^{k-n}\right\|_{L^2}\|\psi_\theta\|_{H^n} \les \|\psi\|_{\Sigma_{T}^\textbf{D}}.
\end{align*}

Hence it suffices to show that
	\begin{align*}
		\begin{aligned} &\|\psi_{\theta}(t)\|_{W^{k,\infty}} \les  t^{-2+\beta+\de} \ve_1.
		\end{aligned}
	\end{align*}
for $t \gg 1$.	To this end, we take a frequency decomposition of $\psi_{\theta}$ such that
	\[
	\psi_{\theta}(t,x)=\sum_{N\in2^{\mathbb{Z}}}\int_{\mathbb{R}^{4}} e^{it\phi_\theta(\xi)}\rho_{N}(\xi) \widehat{f_{\theta}}(\xi)d\xi.
	\]
	Then we get
	\begin{align*}
		t^{2-\beta-\de}\left|\bra{D}^k\psi_{\theta}(t,x) \right|\le\sum_{N\in2^{\mathbb{Z}}}I_{N}(t,x),
	\end{align*}
	where
	\[
	I_{N}(t,x)= t^{2-\beta-\de}\left|\int_{\mathbb{R}^4} e^{it\phi_\theta(\xi)}\bra{\xi}^k \rho_{N}(\xi)\widehat{f_{\theta}}(\xi)d\xi\right|.
	\]
	The proof is based on the stationary phase method. We first decompose the frequency support as
	\[
	\sum_{N\in2^{\mathbb{Z}}}I_{N}(t,x)=\left(\sum_{N\le t^{-1}}+\sum_{N\ge t^{\frac{2}{n-2}}}+\sum_{t^{-1}\le N\le t^{\frac{2}{n-2}}}\right)I_{N}(t,x).
	\]
	The high and low-frequency part can be estimated as
\begin{align*}
	\left(\sum_{N\le t^{-1}} + \sum_{N\ge t^{\frac{2}{n-2}}}\right)I_{N}(t,x)&\les  \left(\sum_{N\le t^{-1}} + \sum_{N\ge t^{\frac{2}{n-2}}}\right)t^{2-\beta-\de} \bra{N}^k\|\rho_{N}\|_{L_{\xi}^{2}}\|\rho_{N}\widehat{f_{\theta}}\|_{L_{x}^{2}}\\
	&\les  \left(\sum_{N\le t^{-1}} + \sum_{N\ge t^{\frac{2}{n-2}}}\right)t^{2-\beta-\de}N^2 \bra{N}^{k-n}\|\psi_\theta\|_{H^n} \les \ve_1.
\end{align*}

	Let us now focus on the mid-frequency part. To prove this part, we use both the non-stationary and stationary
	phase methods. 
	From \eqref{theta-phase}, we have
	\[
	\nabla_{\xi}\phi_\theta(\xi)=-\theta\frac{\xi}{\bra{\xi}}+\frac{x}{t}.
	\]
	Then, when $|x|\ge t$, the phase $\phi_{\theta}$ is non-stationary. Indeed, we have
	\begin{align}
		|\nabla_{\xi}\phi_\theta(\xi)|\gtrsim\left|\frac{|x|}{t}-\frac{|\xi|}{\bra{\xi}}\right|\gtrsim1-\frac{|\xi|}{\bra{\xi}}\gtrsim \bra{\xi}^{-2}.\label{non-stat-1}
	\end{align}
	On the other hand, the phase $\phi_{\theta}$ could be stationary around
	$\xi_{0}$ when $|x|<t$:
	\[
	\nabla_{\xi}\phi_\theta(\xi_{0})=0\;\;\mbox{where}\;\;\xi_{0}=-\theta\frac{x}{\sqrt{t^{2}-|x|^{2}}}.
	\]
	We now set $|\xi_{0}|\sim N_{0} \in 2^\Z$. If $N\nsim N_{0}$, then the phase is non-stationary and simple calculation yields that
	\begin{align}
		\Big|\nabla_{\xi}\phi_\theta(\xi)\Big|\gtrsim\max\left(\frac{|\xi-\xi_{0}|}{\bra{N}^{3}},\frac{|\xi-\xi_{0}|}{\bra{N_{0}}^{3}}\right).\label{non-stat-2}
	\end{align}

Let $\mathbf p_\phi = \frac{\nabla_\xi \phi_\theta}{|\nabla_\xi \phi_\theta|^2}$. Then we write the mid-frequency of
	$I_{N}(t,x)$ by the integration by parts as follows:
	\[
	t^{2-\beta-\de}\int_{\mathbb{R}^{3}} e^{it\phi_{\theta}(\xi)}\bra{\xi}^k \rho_{N}(\xi) \widehat{f_{\theta}}(\xi)d\xi = \sum_{j = 1}^4I_{N}^{j}(t,x),
	\]
	where
	\begin{align}
		\begin{aligned}\label{eq-i123}
			I_{N}^{1}(t,x) & =-\int_{\mathbb{R}^{4}}e^{it\phi_\theta(\xi)}\mathbf p_\phi^2 \cdot \nabla_\xi^2\left(\bra{\xi}^k \rho_N\widehat{f_{\theta}}(\xi)\right)\, d\xi,\\
			I_{N}^{2}(t,x) & =-\int_{\mathbb{R}^{4}}e^{it\phi_\theta(\xi)}\mathbf p_\phi \cdot \Big[(\nabla_\xi \mathbf p_\phi) \nabla_\xi\left(\bra{\xi}^k \rho_N\widehat{f_{\theta}}(\xi)\right)\Big]\, d\xi,\\
     	I_{N}^{3}(t,x) &= -2\int_{\mathbb{R}^{4}}e^{it\phi_\theta(\xi)}(\nabla_\xi \cdot \mathbf p_\phi) \mathbf p_\phi \cdot \nabla_\xi\left(\bra{\xi}^k \rho_N\widehat{f_{\theta}}(\xi)\right)\,d\xi,\\
			I_{N}^{4}(t,x) & =-\int_{\mathbb{R}^{4}}e^{it\phi_\theta(\xi)}\nabla_\xi\cdot [(\nabla_\xi \cdot \mathbf p_\phi) \mathbf p_\phi ]\bra{\xi}^k \rho_N\widehat{f_{\theta}}(\xi)\,d\xi.
		\end{aligned}
	\end{align}

	By \eqref{non-stat-1} and \eqref{non-stat-2}, we obtain the following bounds independent
	of $N_{0}$:
	\begin{align}
		\begin{aligned}\left|\mathbf p_\phi \right| & \les\bra{N}^{3}N^{-1},\\
			\left|\nabla_{\xi}\mathbf p_\phi \right| + \left|\nabla_\xi \cdot \mathbf p_\phi \right| & \les\bra{N}^5N^{-2},\\
			\left|\nabla_\xi\cdot [(\nabla_\xi \cdot \mathbf p_\phi ) \mathbf p_\phi  ]\right| & \les\bra{N}^{7}N^{-3}.
		\end{aligned}
		\label{bound-phase}
	\end{align}
	Using \eqref{bound-phase} and Sobolev embedding, we see that
	\begin{align*}
		&\Big|I_{N}^{1}(t,x)\Big| \\
		& \les t^{-\beta-\de}\bra{N}^{6}N^{-2} \left( \bra{N}^k\left\Vert \nabla_{\xi}^{2}\left(\rho_{N}\widehat{f_{\theta}}\right)\right\Vert _{L_{\xi}^{1}} + \bra{N}^{k-1}\left\Vert \nabla_{\xi}\left(\rho_{N}\widehat{f_{\theta}}\right)\right\Vert _{L_{\xi}^{1}} + \bra{N}^{k-2}\left\Vert \left(\rho_{N}\widehat{f_{\theta}}\right)\right\Vert _{L_{\xi}^{1}} \right)\\
		& \les t^{-\beta-\de}\bra{N}^{6+k}N^{-2}\left(N^2\|x^{2}f_{\theta}\|_{L^2} + N^{2}\left\|  \widehat{ xf_{\theta}}\right\|_{L_{\xi}^{4}} + N^{2-4\ve}\normo{\wh{f_{\theta}}}_{L_{\xi}^{\frac1\ve}}\right)\\
& \les t^{-\beta-\de}\bra{N}^{6+k}N^{-2}\left(N^{2}\|x^{2}f_{\theta}\|_{L^2} + N^{2-4\ve}\normo{\bra{x}^2f_{\theta}}_{L^2}\right)\\
&\les  t^{-\beta} \bra{N}^{6+k} \left( 1+  N^{-4\ve} \right)\ve_1,
	\end{align*}
for sufficiently small $0 < \ve \ll 1$,	which implies that
	\begin{align*}
	\sum_{t^{-1}\le N\le t^{\frac{2}{n-2}}}\Big|I_{N}^{1}(t,x)\Big| &\les  \ve_1,
	\end{align*}
	since $\frac{52}{n-2} \le \beta$. We also obtain, for the  $0 < \ve \ll1$,
	\begin{align*}
		|I_{N}^{2}(t,x)| + |I_{N}^{3}(t,x)| & \les t^{-\beta-\de}\bra{N}^{8+k}N^{-3}\normo{\nabla_{\xi}\left(\rho_{N}\widehat{f_{\theta}}\right)}_{L_{\xi}^{1}}\\
		& \les t^{-\beta-\de}\bra{N}^{8+k}N^{-3}\left(  N^{3}\left\Vert \widehat{xf_{\theta}} \right\Vert _{L_{\xi}^{4}} + N^{3-4\ve}\left\Vert \widehat{f_{\theta}}\right\Vert _{L_{\xi}^{\frac1\ve}} \right)\\
		& \les t^{-\beta}\bra{N}^{8+k}N^{-4\ve}\ve_1.
	\end{align*}
	Then we have
	\begin{align}\label{eq:timeesti-i23}
	\sum_{t^{-1}\le N\le t^{\frac{2}{n-2}}}|I_{N}^{2}(t,x)| + |I_{N}^{3}(t,x)| \les \ve_1.		
	\end{align}
	Note that the estimate \eqref{eq:timeesti-i23} requires that   $\frac{56}{n-2} \le \beta$.   As for $I_{N}^4$, we can estimate the easier way by using  the \eqref{bound-phase}.
	
	It remains to consider the stationary phase case, $N\sim N_{0}$. We
	further decompose $|\xi-\xi_{0}|$ into dyadic pieces $L$. Let  $L_{0}\in2^{\mathbb{Z}}$
	satisfy that $\frac{L_{0}}{2}<t^{-1}\le L_{0}$. Then we write
	\[
	t^{2-\beta-\de}\left|\int_{\mathbb{R}^{4}}e^{it\phi_\theta(\xi)}\bra{\xi}^k\rho_{N}(\xi)  \widehat{f_{\theta}}(\xi)d\xi\right|\le\sum_{L=L_{0}}^{2^{10}N}|J_{L}|,
	\]
	where
	\begin{align*}
		J_{L}(t,x):= \left\{\begin{aligned} & t^{2-\beta-\de}\int_{\mathbb{R}^{4}}e^{it\phi_\theta(\xi)}\bra{\xi}^k\rho_{\le L_{0}}(\xi-\xi_{0})\rho_{N}(\xi) \widehat{f_{\theta}}(\xi)d\xi\qquad\mbox{when }L=L_{0},\\
			& t^{2-\beta-\de}\int_{\mathbb{R}^{4}}e^{it\phi_\theta(\xi)}\bra{\xi}^k \rho_{L}(\xi-\xi_{0})\rho_{N}(\xi) \widehat{f_{\theta}}(\xi)d\xi\quad\qquad\mbox{when }L>L_{0}.
		\end{aligned}\right.
	\end{align*}
	The case $L=L_{0}$ is the stationary one. The fact $L_{0}\sim t^{-1}$ implies that 
	\[
	|J_{L_{0}}|\les t^{2-\beta-\de} L_{0}^{2}\|\rho_{N}\widehat{f_{\theta}}\|_{L_{\xi}^{2}}\les  \ve_1.
	\]
	The case $L>L_{0}$ returns to non-stationary phase. By the  integration by parts, we write $J_{L}(t,x)$ as in \eqref{eq-i123}:
	\[
	J_{L}(t,x) = \sum_{j = 1}^4 J_{L}^{j}(t,x),
	\]
	where
	\begin{align*}
		\begin{aligned}J_{L}^{1}(t,x) & =-t^{-\beta-\de}\int_{\mathbb{R}^{4}}e^{it\phi_\theta(\xi)}\mathbf p_\phi ^2 \cdot \nabla_\xi^2\left(\bra{\xi}^k\rho_{L}(\xi-\xi_{0})\rho_{N}(\xi)\widehat{f_{\theta}}(\xi)\right)d\xi,\\
			J_{L}^{2}(t,x) &
=-t^{-\beta-\de}\int_{\mathbb{R}^{4}}e^{it\phi_\theta(\xi)}\mathbf p_\phi  \cdot \Big[(\nabla_\xi \mathbf p_\phi ) \nabla_\xi\left(\bra{\xi}^k\rho_{L}(\xi-\xi_{0})\rho_{N}(\xi)\widehat{f_{\theta}}(\xi)\right)\Big]d\xi,\\
			J_{L}^{3}(t,x) & =-2t^{-\beta-\de}\int_{\mathbb{R}^{4}}e^{it\phi_\theta(\xi)}(\nabla_\xi \cdot \mathbf p_\phi ) \mathbf p_\phi  \cdot \nabla_\xi\left(\bra{\xi}^k\rho_{L}(\xi-\xi_{0})\rho_{N}(\xi)\widehat{f_{\theta}}(\xi)\right)d\xi,\\
			J_{L}^{4}(t,x) & =-t^{-\beta-\de}\int_{\mathbb{R}^{4}}e^{it\phi_\theta(\xi)}\nabla_\xi\cdot [(\nabla_\xi \cdot \mathbf p_\phi ) \mathbf p_\phi  ]\bra{\xi}^k\rho_{L}(\xi-\xi_{0})\rho_{N}(\xi)\widehat{f_{\theta}}(\xi)d\xi.
		\end{aligned}
	\end{align*}
	Each $J_{L}^{j}$ can be estimated similarly to $I_N^j$ with the following phase
	bounds:
	\begin{align*}
		\begin{aligned}\left|\mathbf p_\phi \right| & \les L^{-1}\bra{N}^{3},\\
			\left|\nabla_\xi \mathbf p_\phi \right| + \left|\nabla_\xi \cdot \mathbf p_\phi \right|& \les L^{-1}\bra{N}^{2}+L^{-2}\bra{N}^{3},\\
			\left|\nabla_\xi\cdot [(\nabla_\xi \cdot \mathbf p_\phi ) \mathbf p_\phi  ]\right| & \les L^{-2}\bra{N}^{6}+L^{-3}\bra{N}^{5}.
		\end{aligned}
	\end{align*}
	This finishes the proof of time decay of Dirac part \eqref{eq:timedecay-d}.
\end{proof}

	\begin{prop}[Time decay of amplitudes for Maxwell part]\label{prop:timedecay-maxwell}
		Let $\mu = 0, \cdots, 4$ and $ \kappa \in \{+,-\}$. Assume
		that  $A_{\mu}$ satisfies a priori assumption \eqref{eq:assumption-apriori},
		for given $\ve_{1}$ and $T$.  Then there exists $C$
		satisfying that for $0 \le s \le 20$,
		\begin{align}
			\|A_{\mu, \kappa}(t)\|_{W^{s,\infty}}&\le C\bra{t}^{-\frac32}\ve_{1}, \label{eq:timedecay-m}
		\end{align}
		for any $t \in [0, T]$.
	\end{prop}
	
\begin{proof} 	 By the same argument as the proof of time decay of spinor, we may assume that $t \gg 1$. We will show that
	\begin{align*}
			&\|A_{\mu,\ka}(t)\|_{W^{s,\infty}} \les t^{-\frac32}\ve_1,
	\end{align*}
	for $\ka \in \{ +,-\}$. We take a frequency decomposition for $A_{\mu,\ka}$ as
	\[
	t^\frac32\bra{D}^s A_{\mu, \ka}(t,x) \les \sum_{N \in 2^\Z} t^\frac32\left| \int_{\R^4} e^{it\vp_{\ka}(\xi)} \bra{\xi}^s \rho_N(\xi) \wh{F_{\mu, \ka}}(t,\xi)  d\xi \right| =: \sum_{N\in 2^\Z} \wt{I}_N,
	\]
	where
	\[
	\vp_{\ka}(\xi) := -|\xi| + \ka \xi \cdot \frac{x}{t}.
	\]
	
	For the low and high frequency parts, we get, for $\zeta := \frac\beta4 + \frac\de4$,
	\begin{align*}
	\sum_{N \le t^{-\frac34-\zeta}} \wt{I}_N(t,x) &\les \sum_{N \le t^{-\frac34-\zeta} }t^{\frac32} \|\rho_N\|_{L_\xi^2} \normo{ {F_{\mu, \ka}} }_{L^2} \les  \ve_1
	\end{align*}
	and
	\[
	\sum_{N \ge t^{\frac2{n-2}}} \wt{I}_N(t,x) \les \sum_{N \ge t^{\frac2{n-2}}}t^\frac32 N^{s-n + 2} \|A_{\mu, \ka}\|_{ H^{n}} \les t^{-\frac12+\frac{2s}{n-2}}\|A_{\mu, \ka}\|_{ H^{n}} \les \ve_1.
	\]
	We now consider the mid frequencies $t^{-\frac34-\zeta} \le N \le t^{\frac 2{n-2}} $. To handle this, we further divide an angle $\left( 1- \ka \frac{\xi\cdot x}{|\xi|t}\right)$ dyadically with $L\in 2^{\Z}$. Let us choose dyadic number $L_0 \sim t^{-\frac58}$. Then we write
	\[
	\wt{I}_N(t,x) \le \sum_{L\ge L_0} \wt{I}_{N,L}(t,x),
	\]
	where
	\begin{align*}
		\wt{I}_{N,L}(t,x) = \left\{\begin{aligned}
			&t^\frac32\left|\int_{\R^4} e^{it\vp_{\ka}(\xi)} \bra{\xi} ^s \rho_{\le L_0}\left( 1- \ka\frac{\xi\cdot x}{|\xi|t}\right) \rho_N(\xi) \wh{F_{\mu, \ka}}(\xi) d\xi \right|  \qquad \mbox{for}\;\; L= L_0,\\
			&t^\frac32\left| \int_{\R^4} e^{it\vp_{\ka}(\xi)} \bra{\xi} ^s \rho_{L}\left( 1- \ka\frac{\xi\cdot x}{|\xi|t}\right) \rho_N(\xi) \wh{F_{\mu, \ka}}(\xi) d\xi \right| \qquad\quad \mbox{for}\;\; L> L_0.
		\end{aligned}\right.
	\end{align*}
	If $L = L_0$, then H\"older inequality and Sobolev embedding yield that
	\begin{align*}
	\wt{I}_{N,L_0}(t,x) &\les t^\frac32\bra{N}^s\|\rho_N \rho_{\le L_0}\|_{L_\xi^\frac1{1-\ve}} \left\|\rho_N \wh{F_{\mu,\ka}}\right\|_{L_\xi^\frac1\ve} \les t^\frac32\bra{N}^sN^{4-4\ve}L_0^{3-3\ve} \left\|\rho_{N} \wh{F_{\mu, \ka}}\right\|_{L_\xi^\frac1\ve} \\
	&\les  t^{-\frac{15}{8}+ \frac{15}8\ve}N^{2-4\ve}\bra{N}^{s} \left(  \left\| \wh{x^2F_{\mu, \ka}}\right\|_{\doth^2} + \left\| \wh{xF_{\mu, \ka}}\right\|_{\doth^1} + \left\|\wh{F_{\mu, \ka}}\right\|_{L^2} \right) \\
	&\les t^{-\frac{15}{8}+ \frac{15}8\ve + \frac\beta2 + \frac\de2} N^{2-4\ve}\bra{N}^{s} \ve_1,
		\end{align*}
		where we choose $\ve>0$, satisfying $0<\ve \ll \frac38 -\frac{\beta}{2}-\frac\de2 $, since $\frac{\beta}{2}+\frac\de2 \le \frac14$. Then this leads to that
	\[
	\sum_{t^{-\frac34-\zeta} \le N \le t^{\frac2{n-2}}}\wt{I}_{N,L_0}(t,x) \les \ve_1.
	\]
	We now consider the case $L > L_0$. Note that
	\[
	r\partial_r \vp_{\ka} = \vp_{\ka}, \quad \partial_r\left(\frac{r}{\vp_{\ka}}\right) = 0,
	\]
	where $\partial_r f = \frac{\xi}{|\xi|}\cdot\nabla_\xi{f} $ is the radial derivative. Taking the integration by parts twice, we see that
	\begin{align*}
		\wt{I}_{N,L}(t,x) &\les t^{-\frac12}  \left| \int_{\R^4}  e^{it\vp_{\ka}(\xi)}  \left( \partial_r \vp_{\ka}\right)^{-2} \partial_r^2\Big[ \bra{\xi} ^s\rho_L\left( 1- \ka\frac{\xi\cdot x}{|\xi|t}\right) \rho_N(\xi) \wh{F_{\mu, \ka}}(\xi) \Big] d\xi	\right|.
	\end{align*}
	Since
	\[
	|\partial_r \vp_{\ka}| = |\vp_{\ka}/r| = \left| 1 - \ka\frac {\xi\cdot x}{t|\xi|} \right|  \sim L,
	\]
by H\"older inequality we see that
	\begin{align*}
		\wt{I}_{N,L}(t,x) &\les t^{-\frac12}  L^{-2} \Big\|\partial_r^2 \left(\bra{\xi} ^s\rho_L \rho_N \wh{F_{\mu, \ka}} \right) \Big\|_{L_\xi^1}\\
		&\les t^{-\frac12}L^{-2}\bra{N}^{s} \left(\|\widetilde\rho_N\rho_L \|_{L_\xi^2} \left\|P_N\left(x^2 F_{\mu, \ka} \right)\right\|_{L^2} + N^{-1} \|\partial_r\widetilde\rho_N\rho_L\|_{L_\xi^2}\left\||\xi|\rho_N\wh{ x F_{\mu, \ka}}\right\|_{L_\xi^2}  \right. \\
		&\hspace{8cm} \left. + \left\|\partial_r^2\widetilde\rho_N\rho_L\right\|_{L_\xi^2}\|\wh{ P_N F_{\mu, \ka}}\|_{L_\xi^2}\right)\\
		&\les t^{-\frac12} L^{-2}\bra{N}^s \left( L^\frac32 \|x^2F_{\mu, \ka}\|_{\dot{H}^2} +  L^\frac32\left\|  x F_{\mu, \ka}\right\|_{\doth^1}  +  L^\frac32 \| F_{\mu, \ka}\|_{L^2}\right) \\
		&\les  t^{-\frac38} \bra{N}^s  L^{-\frac12} \ve_1, 
	\end{align*}
	which implies that
	\begin{align}\begin{aligned}\label{eq:l2-decay}
		\sum_{t^{-\frac43-\zeta} \le N \le t^\frac2{n-2}}\sum_{L_0 <L \le1 } \wt{I}_{N,L}(t,x) &\les t^{-\frac1{16}+\frac{s}{n-2}} \ve_1,\\
        \sum_{t^{-\frac43-\zeta} \le N \le t^\frac2{n-2}}\sum_{L >1 }\wt{I}_{N,L}(t,x)  &\les t^{-\frac38+ \frac{s}{n-2}} \ve_1.
	\end{aligned}\end{align}
To conclude \eqref{eq:l2-decay}, it  requires the condition $n \ge 320$.

We now treat the case $|x| \le 10t$. Similarly, we decompose $\left|  \frac{\xi}{|\xi|}- \frac{ x}{t} \right|$ into $L\in 2^{\Z}$. Let $L_0 \sim t^{-\frac58}$. Then $\wt{I}_N$ is bounded by the following terms:
\[ 
\wt{I}_N(t,x) \le \sum_{L\ge L_0} \wt{J}_{N,L}(t,x)
\]
where
\begin{align*}
	\wt{J}_{N,L}(t,x) = \left\{\begin{aligned}
		&t^{\frac32}\left|\int_{\R^4} e^{it\vp_{\ka}(\xi)} \bra{\xi}^s\rho_{\le L_0}\left( \frac{\xi}{|\xi|}- \frac{ x}{t}\right) \rho_N(\xi) \wh{F_{\mu, \ka}}(\xi) d\xi \right|  \qquad \mbox{for}\;\; L= L_0,\\
		&t^\frac32\left| \int_{\R^4} e^{it\vp_{\ka}(\xi)} \bra{\xi}^s \rho_{L}\left( \frac{\xi}{|\xi|}- \frac{ x}{t}\right) \rho_N(\xi) \wh{F_{\mu, \ka}}(\xi) d\xi \right| \qquad\quad \mbox{for}\;\; L> L_0.
	\end{aligned}\right.
\end{align*}
For the case $L=L_0$, we may estimate in a similar way to the estimate of $\wt{I}_{N,L_0}$. Then we assume $L > L_0$. By integration parts in $\xi$, we write $\wt{J}_{N,L} = \sum_{\ell = 1}^4 \wt{J}_{N, L}^\ell$ with $\mathbf p_\vp = \frac{\nabla_\xi\vp_{\ka}}{|\nabla_\xi\vp_{\ka}|^2}$, where
	\begin{align*}
	\wt{J}_{N,L}^1(t,x) &\les t^{-\frac12}  \left| \int_{\R^4}  e^{it\vp_{\ka}(\xi)} \mathbf p_\vp^2  \nabla_\xi^2\Big[ \bra{\xi}^s\rho_L\left(\frac{\xi}{|\xi|}- \frac{ x}{t}\right) \rho_N(\xi) \wh{F_{\mu, \ka}}(\xi) \Big]\, d\xi	\right|,\\
	\wt{J}_{N,L}^2(t,x) &= t^{-\frac12} \left| \int_{\R^4} e^{it\vp_{\ka}(\xi)} \mathbf p_\vp \cdot \Big[(\nabla_\xi \mathbf p_\vp) \nabla_\xi\left( \bra{\xi}^s\rho_L\left(\frac{\xi}{|\xi|}- \frac{ x}{t}\right) \rho_N(\xi) \wh{F_{\mu, \ka}}(\xi) \right)\Big]\, d\xi\right|	,\\
	\wt{J}_{N,L}^3(t,x) &= t^{-\frac12} \left| \int_{\R^4} e^{it\vp_{\ka}(\xi)} (\nabla_\xi \cdot \mathbf p_\vp) \mathbf p_\vp \cdot \nabla_\xi\left( \bra{\xi}^s\left(\frac{\xi}{|\xi|}- \frac{ x}{t}\right) \rho_N(\xi) \wh{F_{\mu, \ka}}(\xi) \right)\, d\xi\right|	,\\
	\wt{J}_{N,L}^4(t,x) &= t^{-\frac12} \left| \int_{\R^4} e^{it\vp_{\ka}(\xi)} \nabla_\xi\cdot [(\nabla_\xi \cdot \mathbf p_\vp) \mathbf p_\vp ]  \bra{\xi}^s\rho_L\left(\frac{\xi}{|\xi|}- \frac{ x}{t}\right) \rho_N(\xi) \wh{F_{\mu, \ka}}(\xi)\,  d\xi \right|.
\end{align*}
One can use the same arguments as above and the following phase estimates
\begin{align*}
	|\mathbf p_\vp| & \les L^{-1}\\
	|\nabla_\xi \mathbf p_\vp| + |\nabla_\xi \cdot \mathbf p_\vp|  &\les L^{-3} N^{-1} \\
	|\nabla_\xi\cdot [(\nabla_\xi \cdot \mathbf p_\vp) \mathbf p_\vp ]| &\les L^{-4}N^{-2} + L^{-3}N^{-2}.
\end{align*}
This completes the proof of \eqref{eq:timedecay-m}.
\end{proof}

\section{Useful estimates}\label{sec:useful}
In this section, we gather several lemmas that are used throughout the paper. First we introduce Coifman-Meyer estimates, utilized to handle the multipliers in the energy estimates.

\begin{lemma}[Coifman-Meyer operator estimates]\label{lem:coif-mey}
	Let $1 \le p, q \le \infty$ satisfy that $\frac{1}{p}+\frac{1}{q}=\frac{1}{2}$. Assume that
	\[
	\|\textbf{m}\|_{\rm CM}:=\left\Vert \iint_{\mathbb{R}^{4+4}}\mathbf{m}(\xi,\zeta)e^{ix\cdot\xi}e^{iy\cdot\eta}d\xi d\eta\right\Vert _{L_{x,y}^{1}}\le C_{\mathbf{m}}\quad(\zeta = \eta \;\,\mbox{or}\,\;\xi-\eta ).
	\]
	Then
	\begin{align*}
		\left\Vert \int_{\mathbb{R}^{4}}\mathbf{m}(\xi, \zeta)\widehat{v}(\eta)\widehat{w}(\xi-\eta)d\eta\right\Vert _{L_{\xi}^{2}}\les C_{\mathbf{m}}\|v\|_{L^{p}}\|w\|_{L^{q}}.
	\end{align*}
\end{lemma}

\begin{rem}\label{rem:coif-meyer}
	In the energy estimates, we focus on the following oscillatory integral of the form
	\begin{align}\label{eq:rem-oscil}
		\int_{\R^4} m(\xi,\eta) \wh{f}(\eta) \wh{g}(\xi-\eta)    d\eta.
	\end{align}
	Lemma \ref{lem:coif-mey} implies that the above integral has the same boundedness properties to the one obtained by H\"older inequality in space. If the multiplier $m(\xi,\eta)$ in \eqref{eq:rem-oscil} satisfies that
	\begin{align*}
		|\nabla_\xi^n \nabla_\eta^m m | \les |\xi|^{-n}|\eta|^{-m},
	\end{align*} 
	it leads us to  $	\|m\|_{\cm} \les 1$. 	This approach will be used crucially to handle the multipliers, coming from the various resonances in our main estimates.
\end{rem}

The following two lemmas are easily obtained.
\begin{lemma}[Hardy-Littlewood-Sobolev inequality]\label{lem:hls}
	Let $u_j : \R^{4} \to \C^4$. Then
	\begin{align}\label{eq:hls}
		\||x|^{-k} * \bra{u_1,u_2} \|_{L^p} \les  \|u_1\|_{L^q}\|u_2\|_{L^r},
	\end{align}
	for $\frac1p + \frac k4 = \frac1q + \frac1r$ satisfying $1 \le p <\infty$ and $0<k < 4$.
\end{lemma}

\begin{lemma}\label{lem:projection-deri}
	Let $\xi \in \mathbb{R}^{4}$ and  $\theta \in \{+,-\}$. 	Then we have
	\begin{align*}
		\Big|\nabla_{\xi}^{n}\Pi_{\theta}(\xi)\Big|\les \bra{\xi}^{-n},
	\end{align*}
	for $n \ge 0$.
\end{lemma}

\begin{lemma}\label{lem:time-derivative}
	Let $\theta, \ka \in \{+,-\}$. Assume that $\psi,A_{\mu} \in C([0,T],H^n)$ satisfy a priori assumption \eqref{eq:assumption-apriori}. Then
	\begin{align}
		\|\p_t f_\theta(t)\|_{H^{20}} & \les  \bra{t}^{-2+\frac{3\beta}2+\frac{3\de}2}\ve_1^2, \label{eq:esti-timederivative-d}\\ 
	\|\p_t F_{\mu,\ka}(t)\|_{H^{20}} & \les \bra{t}^{-1+\frac\beta2 +\frac\de2}\ve_1^2, \label{eq:esti-timederivative-m}
	\end{align}
where $f_\theta,\, F_{\mu,\ka}$ are defined in \eqref{eq:interaction}.
\end{lemma}
\begin{proof}
	Let us consider \eqref{eq:esti-timederivative-d}. By Duhamel's principle \eqref{eq:profile-dirac}, we have
	\begin{align*}
		f_{\theta}(t) &= \psi_{0, \theta} + i\sum_{\theta_1, \theta_2 \in \{\pm \}}\int_0^t e^{\theta is\brad}\Pi_{\theta}(D)\big(A_{\mu, \theta_2} \alpha^\mu \psi_{\theta_1}\big)(s)\,ds,
	\end{align*}
	which implies that
	\begin{align*}
		\p_t f_{\theta}(t) &=  i\sum_{\theta_1, \theta_2 \in \{\pm \}} e^{\theta it\brad}\Pi_{\theta}(D)\big(A_{\mu, \theta_2} \alpha^\mu \psi_{\theta_1}\big)(t).
	\end{align*}
	Then we estimate
	\begin{align*}
		\|\p_t f_\theta(t)\|_{H^{20}} \les \|A_{\mu,\thet}(t)\|_{H^{20}} \|\psi_\theo(t)\|_{W^{20,\infty}} \les \bra{t}^{-2+\frac{3\beta}2+\frac{3\de}2}\ve_1^2.
	\end{align*}
	
	Let us move on to the proof of \eqref{eq:esti-timederivative-m}. Similarly, by \eqref{eq:profile-maxwell}, direct calculation leads us that
	\begin{align*}
		\p_tF_{\mu, \ka}(t) &=  -\ka \frac{i}2\sum_{\ka_1, \ka_2 \in \{\pm\}} e^{-\ka it|D|}|D|^{-1}\bra{\psi_{\ka_1}(t),\alpha_\mu\psi_{\ka_2}(t)}.
	\end{align*}
	Using \eqref{eq:hls}, one gets that
	\begin{align*}
		\normo{\p_tF_{\mu, \ka}(t)}_{H^{20}} \les \normo{\psi_\kao(t)}_{L^4}\normo{\psi_\kat(t)}_{H^{20}} \les \bra{t}^{-1+\frac\beta2 + \frac\de2}\ve_1^2,
	\end{align*}
	and the desired conclusion follows.
\end{proof}

\begin{lemma}\label{lem:1st-weight-high}
Let $\theta, \ka \in \{+,-\}$. Assume that $\psi,A_\mu \in C([0,T],H^n)$ satisfies a priori assumption \eqref{eq:assumption-apriori}. Then
	\begin{align}
		\|xf_\theta\|_{H^{20}} &\les \bra{t}^\de\ve_1,\label{eq:1st-wei-d}\\
		\||D|xF_{\mu,\ka}\|_{H^{20}} &\les \bra{t}^{\frac\beta2 + \frac\de2}\ve_1.\label{eq:1st-wei-m}
	\end{align}
\end{lemma}
\begin{proof}
By Plancherel's theorem, we see that	
	\begin{align*}
		\|xf_\theta\|_{H^{20}} \les \|\left(\nabla_\xi\bra{\xi}^{20}\right)\wh{f_\theta}\|_{L_\xi^2} + \normo{\nabla_\xi \left(\bra{\xi}^{20}\wh{f_\theta}\right)}_{L_\xi^2}.
	\end{align*}
Since the first term is bounded by $\normo{\psi_\theta}_{H^n}$, we focus on the second term.
\begin{align*}
	&\int_{\R^4} \bra{\nabla_\xi \left(\bra{\xi}^{20}\wh{f_\theta}\right),\nabla_\xi \left(\bra{\xi}^{20}\wh{f_\theta}\right)}   d\xi\\
	 &= \int_{\R^4} \bra{(-\Delta_\xi) \left(\bra{\xi}^{20}\wh{f_\theta}\right), \bra{\xi}^{20}\wh{f_\theta}}   d\xi\\
	&=\int_{\R^4} \bra{\left[(-\Delta_\xi) \bra{\xi}^{20}\right]\wh{f_\theta}, \bra{\xi}^{20}\wh{f_\theta}}   d\xi + \int_{\R^4} \bra{\nabla_\xi \bra{\xi}^{20}\cdot\nabla_\xi\wh{f_\theta}, \bra{\xi}^{20}\wh{f_\theta}}   d\xi \\
	&\hspace{7cm} 	+\int_{\R^4} \bra{-\Delta_\xi\wh{f_\theta}, \bra{\xi}^{40}\wh{f_\theta}}   d\xi\\
	& \les  \|\psi_\theta\|_{H^{20}}^2+ \|xf_\theta\|_{L^2}\|\psi_\theta\|_{H^{40}} + \|x^2 f_\theta\|_{L^2}\|\psi_\theta\|_{H^{40}}\\
	& \les  \ve_1^2+ \bra{t}^{\de}\ve_1^2.
\end{align*}
Since $0<\ve_1<1$, this completes the proof of \eqref{eq:1st-wei-d}.

	Let us move on to \eqref{eq:1st-wei-m} Similarly,
	\begin{align*}
		\||D|xF_{\mu,\ka}\|_{H^{20}} \les \|\left(\nabla_\xi\bra{\xi}^{20}|
		\xi|\right)\wh{F_{\mu,\ka}}\|_{L_\xi^2} + \normo{\nabla_\xi \left(\bra{\xi}^{20}|\xi|\wh{F_{\mu,\ka}}\right)}_{L_\xi^2},
	\end{align*}
	and the first term can be bounded directly. For the second term, we have
	\begin{align*}
		&\int_{\R^4} \bra{\nabla_\xi \left(\bra{\xi}^{20}|\xi|\wh{F_{\mu,\ka}}\right),\nabla_\xi \left(\bra{\xi}^{20}|\xi|\wh{F_{\mu,\ka}}\right)}   d\xi\\
		&=\int_{\R^4} \bra{\left[(-\Delta_\xi) \bra{\xi}^{20}|\xi|\right]\wh{F_{\mu,\ka}}, \bra{\xi}^{20}|\xi|\wh{F_{\mu,\ka}}}   d\xi + \int_{\R^4} \bra{\nabla_\xi( \bra{\xi}^{20}|\xi|)\cdot\nabla_\xi\wh{F_{\mu,\ka}}, \bra{\xi}^{20}|\xi|\wh{F_{\mu,\ka}}}   d\xi \\
		&\hspace{7cm} 	+\int_{\R^4} \bra{-|\xi|^2\Delta_\xi\wh{F_{\mu,\ka}}, \bra{\xi}^{40}\wh{F_{\mu,\ka}}}   d\xi\\
		& \les  \|A_{\mu,\ka}\|_{H^{n}}^2+ \|xF_{\mu,\ka}\|_{\doth^1}\|A_{\mu,\ka}\|_{H^{n}} + \|x^2 F_{\mu,\ka}\|_{\doth^2}\|A_{\mu,\ka}\|_{H^{n}}\\
		& \les   \bra{t}^{\frac\beta2+\frac\de2}\ve_1^2.
	\end{align*}
	This finishes the proof of \eqref{eq:1st-wei-m}.
\end{proof}

\section{Proof of the main theorem}\label{sec:mainproof}

Given any $T > 0$, let $\psi, A_{\mu}$ be solutions with initial data satisfying \eqref{condition-initial} on $[0,T]$. For some small $0< \varepsilon_1 <1$ and $\theta,\ka \in \{ +,-\}$ we assume that
	\[
	\|\psi_{\theta}\|_{\Sigma_{T,\theta}^\textbf{D}}, \|A_{\mu, \ka}\|_{\Sigma_{T,\ka}^\textbf{M}} \le K\ve_{1}.
	\]
	Then there exists $C$ depending only on $K$ such that
	\begin{align}
		\|\psi_{\theta}\|_{\Sigma_{T,\theta}^\textbf{D}}, \|A_{\mu, \ka}\|_{\Sigma_{T,\ka}^\textbf{M}} \le \ve_{0} + C \ve_{1}^{2}.\label{claim}
	\end{align}
	From \eqref{claim}, Propositions \ref{prop:timedecay} and \ref{prop:timedecay-maxwell} exhibit the global bound \eqref{thm:decay}. 	The claim \eqref{claim} will turn out to follow from Propositions \ref{prop:energy-esti-d} and \ref{prop:energy-esti-m} below. Then the a priori estimate with a bootstrap argument yields the global existence for sufficiently small $\ve_0$. Moreover, by the time reversibility, we obtain a solution for all times.
	
Now we introduce the following energy estimates.
\begin{prop}\label{prop:energy-esti-d} Let  $\theta \in \{ + , - \}$. Assume
that  $\psi, A_{\mu} \in C([0,T], H^n)$ satisfy the a priori assumption \eqref{eq:assumption-apriori}.  Let $\psi_{\theta} \in C([0,T], H^n)$ be a solution to \eqref{eq:half-dirac} with initial data satisfying \eqref{condition-initial}. Then we obtain the following
estimates:
\begin{align}
  \sup_{t\in[0,T]}\|\psi_{\theta}(t)\|_{H^{n}} &\le \ve_{0} + C\ve_{1}^2,\label{eq:high-d}\\
 \sup_{t\in[0,T]}\|x f_{\theta}(t)\|_{L^{2}} &\le \ve_{0} + C\ve_{1}^{2},\label{eq:1st-moment-d}\\
  \sup_{t\in[0,T]} \bra{t}^{-\de}\| x^{2}f_{\theta}(t)\|_{L^2} &\le \ve_{0} + C\ve_{1}^{2}.\label{eq:2nd-moment-d}
\end{align}
\end{prop}

\begin{prop}\label{prop:energy-esti-m}
Let $\ka \in \{+,-\}$. Assume
that $\psi, A_\mu$ satisfy the a priori bound \eqref{eq:timedecay-d} for given $\ve_{1}$ and $T$. Let $A_{\mu, \ka} \in C([0,T], H^n)$ be a solution to \eqref{eq:half-maxwell} with initial data satisfying \eqref{condition-initial}. Then for a sufficiently small $\ve_1$ there exists $C$
satisfying that if $1 \le \ell \le n$, then
\begin{align}
 \sup_{t\in[0,T]} \left[\bra{t}^{-\frac\beta2 - \frac\de2}\|A_{\mu, \ka}(t)\|_{H^n} + \| A_{\mu, \ka}(t)\|_{\dot H^\ell} \right] &\le \ve_0 + C\ve_1^2,\label{eq:high-m}\\
\sup_{t\in[0,T]} \bra{t}^{-\frac\beta2 -\frac\de2}\|x F_{\mu, \ka}(t)\|_{\dot H^1}   &\le \ve_{0} + C\ve_{1}^{2},\label{eq:1st-moment-m}\\
 \sup_{t\in[0,T]}\bra{t}^{-\frac18}\| x^{2}F_{\mu, \ka}(t)\|_{\dot H^{2}} &\le \ve_{0} + C\ve_{1}^{2}.\label{eq:2nd-moment-m}
\end{align}
\end{prop}

\subsection{Estimates for high Sobolev part}\label{sec:high} We begin with the estimates of \eqref{eq:high-d} and \eqref{eq:high-m}. In fact, the  energy estimates \eqref{eq:high-d} and \eqref{eq:high-m} can be readily shown as follows: for $\theta \in \{+,-\}$,
\begin{align*}
	\|\psi_{\theta}(t)\|_{H^n} &\le \|\psi_{0, \theta}\|_{H^n} + C\int_0^t \|\left(A_{\mu}\alpha^\mu \psi \right)(s)\|_{H^n}\,ds\\
	&\le \ve_0 + C\int_0^t \left( \left\|A_\mu(s) \right\|_{H^n} \|\psi(s)\|_{L^\infty} + \left\|A_\mu (s) \right\|_{L^\infty} \|\psi(s)\|_{H^n} \right)\,ds\\
	&\le \ve_0 + C\ve_1^2 \int_0^t \left( \bra{s}^{-2 + \frac{3\beta}2 + \frac{3\de}2} + \bra{s}^{-\frac32}  \right)\,ds\\
	&\le \ve_0 + C\ve_1^2.
\end{align*}

Similarly, if $\ell \ge 1$, then we get
\begin{align}\label{eq:l2-maxwell-1}
	\begin{aligned}
	\|A_{\mu, \kappa}(t)\|_{\dot H^\ell} &\le \|a_{ \mu, \kappa}\|_{\dot H^\ell} + C\int_0^t \|\bra{\psi(s),\alpha_\mu \psi(s)}\|_{\dot H^{\ell-1}}\,ds\\
	&\le \ve_0 + C\sum_{k = 0}^{\ell-1}\int_0^t \|\nabla^k\psi\|_{L^\frac{2(\ell-1)}{k}}\|\nabla^{\ell-1-k}\psi\|_{L^\frac{2(\ell-1)}{\ell-1-k}}\,ds\\
	&\le \ve_0 + C\int_0^t \|\psi\|_{H^{\ell-1}}\|\psi\|_{L^\infty}\,ds\\
	&\le \ve_0 + C\ve_1^2\int_0^t\bra{s}^{-2 -\beta- \de}\,ds\\
	&\le \ve_0 + C\ve_1^2
	\end{aligned}
\end{align}
and if $\ell = 0$, then by charge conservation we have
\begin{align}\label{eq:l2-maxwell-2}
	\begin{aligned}
	\|A_{\mu, \ka}(t)\|_{L^2} &\le \|a_{ \mu, \ka}\|_{L^2} + C\int_0^t \|\bra{\psi(s),\alpha_\mu \psi(s)}\|_{L^\frac43}\,ds\\
	&\le \ve_0 + C\int_0^t \|\psi(s)\|_{L^2}^\frac32\|\psi(s)\|_{L^\infty}^\frac12\,ds\\
	&\le \ve_0 + C\ve_1^2\int_0^t \bra{s}^{-1+\frac\beta2+\frac\de2}  \,ds\\
	&\le \ve_0 + C \bra{t}^{\frac\beta2 +\frac\de2}\ve_1^2 .
\end{aligned}
\end{align}
From \eqref{eq:l2-maxwell-1} and \eqref{eq:l2-maxwell-2}, we see that \eqref{eq:high-m}.

\subsection{Scattering results for Maxwell-Dirac system} In this section, we are concerned with the scattering phenomena of Dirac spinor $\psi$ and potential $A_\mu$. Since we obtained decay properties for Dirac and Maxwell parts, respectively, the proof of scattering is quite straightforward. The readers familiar this argument can skip this section. For $t_1 \le t_2 \in [0,T]$, we have
\begin{align*}
\normo{f_\theta(t_2) -f_\theta(t_1)}_{L^2} \les \left\| \int_{t_1}^{t_2} e^{\theta is \bra{D}} A_\mu (s) \al^\mu \psi(s) ds	\right\|_{L^2} \les \int_{t_1}^{t_2} \normo{A_\mu (s) \al^\mu \psi(s)}_{L^2}  ds \les \bra{t_1}^{-\frac12}\ve_1^2.
\end{align*}
We define $\phi_\theta^\infty =  \lim_{t\to \infty} f_\theta(t)$, where the limit is taken in $L^2$. Then we obtain
\begin{align}\label{eq:scattering-dirac}
	\normo{\psi_\theta(t) - e^{-\theta it\bra{D}}\phi_\theta^\infty}_{L^2} \xrightarrow{t \to \infty} 0.
\end{align}
By putting $\psi^\infty(t) = e^{-it\bra{D}}\phi_+^\infty + e^{it\bra{D}}\phi_-^\infty$, this proves the scattering phenomenon for Dirac spinor.

Analogously,  we see that, for $t_1 \le t_2 \in [0,T]$,
\begin{align*}
	\normo{F_{\mu,\theta}(t_2) -F_{\mu,\theta}(t_1)}_{\doth^1} \les \bra{t_1}^{-1+\beta+\de} \ve_1^2,
\end{align*}
which implies that
\begin{align}\label{eq:scattering-maxwell}
	\normo{A_{\mu,\theta}(t) - e^{\theta it|D|}B_{\mu,\theta}^\infty}_{\doth^1} \xrightarrow{t \to \infty} 0,
\end{align}
by setting $B_{\mu,\theta}^\infty := \lim_{t \to \infty} F_{\mu,\theta}(t)$ in $\doth^1$. Defining $A_{\mu}^\infty(t):= e^{it|D|}B_{\mu,\theta}^\infty + e^{-it|D|}B_{\mu,\theta}^\infty$, we estimate
\begin{align*}
	\normo{A_{\mu}(t) - A_{\mu}^\infty(t)}_{\doth^1} + \normo{\p_t A_{\mu}(t) -\p_t A_{\mu}^\infty(t)}_{L^2} \xrightarrow{t \to \infty} 0.
\end{align*}

\begin{rem}\label{rem:scattering}
	By similar estimates in Section \ref{sec:high}, we readily obtain  \eqref{eq:scattering-dirac} and \eqref{eq:scattering-maxwell} in $H^n$ and $\doth^n$, instead of $L^2$ and $\doth^1$, respectively. These deduce that scattering results for \eqref{md} can be obtained in $H^n \times \doth^n \times \doth^{n-1}$.
\end{rem}

\section{First order weighted energy estimates for spinor: Proof of \eqref{eq:1st-moment-d}}\label{sec:1st-dirac}

In order to estimate $xf_\theta$ in $L^2$, we use Plancherel's theorem as follows:
\[
\normo{x f_\theta(t)}_{L^2} = \normo{ \nabla_\xi \wh{f_\theta}(t,\xi)}_{L_\xi^{2}}.
\]
This implies that, for $\thez \in \{+,-\}$,
\begin{align}
	\nabla_{\xi} \wh{f_\thez}(t,\xi) = & \sum_{\theo,\thet \in \{\pm\}}\int_0^t  \int_{\R^4} e^{isp_\Theta(\xi,\eta)}  \Pi_{\thez}(\xi) \wh{F_{\mu,\thet}}(s,\eta) \al^\mu \nabla_\xi\wh{f_{\theo}}(s,\xi-\eta) d\eta ds \label{eq:1st-dirac-1}\\
	&+ \sum_{\theo,\thet \in \{\pm\}}\int_0^t  \int_{\R^4} e^{isp_\Theta(\xi,\eta)}  \nabla_\xi \Pi_{\thez}(\xi) \wh{F_{\mu,\thet}}(s,\eta) \al^\mu \wh{f_{\theo}}(s,\xi-\eta) d\eta ds \label{eq:1st-dirac-2}\\
	&+ \sum_{\theo,\thet \in \{\pm\}} \int_0^t s \int_{\R^4} e^{isp_\Theta(\xi,\eta)} \nabla_\xi p_\Theta(\xi,\eta) \Pi_{\thez}(\xi) \wh{F_{\mu,\thet}}(s,\eta) \al^\mu \wh{f_{\theo}}(s,\xi-\eta) d\eta ds,\label{eq:1st-dirac-3}
\end{align}
where $p_\Theta$ is defined in \eqref{eq:phase-dirac}. For the estimate of \eqref{eq:1st-dirac-1}, by Plancherel's theorem, we use  H\"older inequality and  \eqref{eq:timedecay-m} to obtain that
\begin{align*}
	\|\eqref{eq:1st-dirac-1}\|_{L_\xi^2} &\les \int_0^t \normo{A_{\mu,\thet}(s)e^{-\theo it\bra{D}}xf_\theo(s) }_{L^2}  ds \\
	&\les \int_0^t \|A_{\mu,\thet}(s)\|_{L^{\infty}}\left\|e^{-\theo it\bra{D}}xf_\theo(s)\right\|_{L^2}  ds \les \int_0^t \bra{s}^{-\frac32}\ve_1^2 ds \les  \ve_1^2.
\end{align*}

In view of Lemma \ref{lem:projection-deri}, the derivative on $ \Pi_{\thez}(\xi)$ gives a decay effect in terms of space.  Thus the estimate of \eqref{eq:1st-dirac-2} can turn out more simply than that of \eqref{eq:1st-dirac-1}. 

 Let us move on to \eqref{eq:1st-dirac-3}. Since we need to recover the time growth $s$ in the integrand, this case becomes the most delicate among the contributions of  \eqref{eq:1st-dirac-1}--\eqref{eq:1st-dirac-3}. To obtain the desired bound of \eqref{eq:1st-dirac-2}, we use the space-time resonances which we observed in Section \ref{sec:resonance}. In view of the time resonance \eqref{eq:resonance-time}, we estimate decomposing the case into the time resonance case $\thez = \theo$ and time non-resonance case $\thez \neq \theo$.
\begin{rem}
To estimate \eqref{eq:1st-dirac-3}, we use the space-time resonant approach. Without exploiting the resonance or non-resonance, we directly estimate \eqref{eq:1st-dirac-3} by $L^2 \times L^\infty$ estimates as follows: 
\begin{align*}
	\normo{\eqref{eq:1st-dirac-3}}_{L_\xi^2} \les \int_0^t s \|A_{\mu,\thet}(s)\|_{L^2} \normo{\psi_\theo (s)}_{L^{\infty}} ds \les \int_0^t \bra{s}^{-1+\frac{3\beta}2 +\frac{3\de}2} \ve_1^2 ds \les \bra{t}^{\frac{3\beta}2 +\frac{3\de}2}\ve_1^2.
\end{align*}
In the first inequality, we used Lemma \ref{lem:coif-mey} with $\normo{\nabla_{\xi} p_\Theta}_{\cm} \les 1$ in both cases $\thez=\theo$ and $\thez\neq \theo$. Given a priori assumption \eqref{eq:assumption-apriori}, any time growth has to be not permitted in the estimate of \eqref{eq:1st-dirac-3}. In order to overcome this divergence, we need to get extra time decay by dividing the time resonance and non-resonance. 
\end{rem}

\emph{Time resonance case: $\thez=\theo$.} If $\thez=\theo$, $\nabla_\xi p_\Theta(\xi,\eta)$ possesses a null structure which gets rid of the space-time resonance set. Therefore, we crucially use this null structure to obtain an extra time decay without any singularity. Indeed, using the integration by parts in time with \eqref{eq:resonance-time} and the relation
\begin{align*}
	e^{isp_\Theta(\xi,\eta)} = -i \frac{\p_s e^{isp_\Theta(\xi,\eta)}}{p_\Theta(\xi,\eta)},
\end{align*}
\eqref{eq:1st-dirac-3} can be bounded by the following terms: For all $\theo,\thet \in \{+,-\}$,
	\begin{align}
		&t \int_{\R^4} e^{itp_\Theta(\xi,\eta)} \frac{\nabla_\xi p_\Theta(\xi,\eta)}{p_\Theta(\xi,\eta)} \Pi_{\thez}(\xi) \wh{F_{\mu,\thet}}(t,\eta) \al^\mu \wh{f_{\theo}}(t,\xi-\eta) d\eta ds, \label{eq:1st-dirac-3-1}\\
		&\int_0^t  \int_{\R^4} e^{isp_\Theta(\xi,\eta)} \frac{\nabla_\xi p_\Theta(\xi,\eta)}{p_\Theta(\xi,\eta)} \Pi_{\thez}(\xi) \wh{F_{\mu,\thet}}(s,\eta) \al^\mu \wh{f_{\theo}}(s,\xi-\eta) d\eta ds, \label{eq:1st-dirac-3-2}\\
		&\int_0^t s \int_{\R^4} e^{isp_\Theta(\xi,\eta)} \frac{\nabla_\xi p_\Theta(\xi,\eta)}{p_\Theta(\xi,\eta)} \Pi_{\thez}(\xi) \p_s\wh{F_{\mu,\thet}}(s,\eta) \al^\mu \wh{f_{\theo}}(s,\xi-\eta) d\eta ds,\label{eq:1st-dirac-3-3}\\
		&\int_0^t s \int_{\R^4} e^{isp_\Theta(\xi,\eta)} \frac{\nabla_\xi p_\Theta(\xi,\eta)}{p_\Theta(\xi,\eta)} \Pi_{\thez}(\xi) \wh{F_{\mu,\thet}}(s,\eta) \al^\mu \p_s \wh{f_{\theo}}(s,\xi-\eta) d\eta ds.\label{eq:1st-dirac-3-4}
	\end{align}
Note that simple calculation leads us that
\begin{align}\label{eq:1st-dirac-multi}
	\normo{\frac{\nabla_\xi p_\Theta}{\left(\bra{\xi}+ \bra{\eta}\right)^{20}p_\Theta}  }_\cm \les 1.
\end{align}
\begin{rem}
As described in Section \ref{sec:resonance}, we need to handle the $|\eta|$-singularity in \eqref{eq:1st-dirac-3-1}--\eqref{eq:1st-dirac-3-4} given by the time resonance of $p_\Theta$. Since
\begin{align*}
	|\nabla_\xi p_\Theta| = \left|\frac{\xi}{\bra{\xi}} - \frac{\xi-\eta}{\xi-\eta} \right| \les \frac{|\eta|}{\bra{\xi}+\bra{\xi-\eta}},
\end{align*}	
we can obtain that \eqref{eq:1st-dirac-multi}. See Remark \ref{rem:coif-meyer} for the calculation of \eqref{eq:1st-dirac-multi}.
\end{rem}
 By Lemma \ref{lem:coif-mey} with \eqref{eq:1st-dirac-multi}, we estimate
\begin{align*}
	\left\|\eqref{eq:1st-dirac-3-1}\right\|_{L_\xi^2} \les t \|A_{\mu,\thet}(t)\|_{H^{n}} \|\psi_{\theo}(t)\|_{W^{20,\infty}} \les \bra{t}^{-1+\frac{3\beta}2 + \frac{3\de}2}\ve_1^2 \les \ve_1^2.
\end{align*}
We may obtain the estimate of \eqref{eq:1st-dirac-3-2} similarly. Using Lemma \ref{lem:time-derivative} and Lemma \ref{lem:coif-mey} with \eqref{eq:1st-dirac-multi}, we see that 
\begin{align*}
	&\left\|\eqref{eq:1st-dirac-3-3}\right\|_{L_\xi^2} + \left\|\eqref{eq:1st-dirac-3-4}\right\|_{L_\xi^2}\\ &\les \int_0^t s \left(\normo{e^{-\thet is |D|} \p_s F_{\mu,\thet}(s)}_{H^{20}}\normo{\psi_\theo(s)}_{W^{20,\infty}} + \normo{F_{\mu,\thet}(s)}_{W^{20,\infty}}\normo{e^{-\theo is\bra{D}}\p_s f_\theo(s)}_{H^{20}} \right)\,ds\\
	&\les \int_0^t s \left(\bra{s}^{-3 +\frac{3\beta}2 + \frac{3\de}2} + \bra{s}^{-\frac72 +\frac{3\beta}2 + \frac{3\de}2 } \right)  \ve_1^2 ds \les \ve_1^2,
\end{align*} 
where in the last inequality we also have used  $\beta +\de \le \frac23$.

\emph{Time non-resonance case: $\thez \neq \theo$.} We now treat the time non-resonance $\thez \neq \theo$. Contrast to the resonance case, $\nabla_\xi p_\Theta(\xi,\eta)$ has no null effect in this case. As we observed in \eqref{eq:nonresonance-time}, since this case does not exhibit the time resonance, we can estimate by using the normal form transform without any singularity. For these reasons, we have the same multiplier bound to \eqref{eq:1st-dirac-multi}. Therefore, the proof of estimates for \eqref{eq:1st-dirac-3} can be finished similarly to the proof of time resonance case $\thez = \theo$.

\section{Second-order weighted energy estimates for spinor: Proof of \eqref{eq:2nd-moment-d}}\label{sec:2nd-dirac}
In this section, we devote to prove \eqref{eq:2nd-moment-d}. We begin with 
\[
\normo{x^2 f_\theta(t)}_{L^2} = \normo{ \nabla_\xi^2 \wh{f_\theta}(t,\xi)}_{L_\xi^{2}},
\]
by Plancherel's theorem. In view of the Duhamel's formula for $f_
\theta$ \eqref{eq:duhamel-dirac}, the derivatives $\nabla_\xi^2$ can fall on $e^{isp_\Theta(\xi,\eta)}$, $\Pi_\thez(\xi)$, and $\wh{f_\theo}(\xi)$ and we need to handle each terms case by case. As we estimated in Section \ref{sec:1st-dirac}, the most complicated term appears when all derivatives fall on the phase function $e^{isp_\Theta(\xi,\eta)}$, since we have to recover the time growth $s^2$. The estimates involved in proving the other cases are straightforward, thus we brief the proof for other cases.


\emph{Derivatives fall only on $\Pi_\thez(\xi)$ or $\wh{f_\theo}(\xi)$.} In this case, the derivatives do not fall on the phase function and the derivative on the projection even implies space decay by Lemma \ref{lem:projection-deri}. Thus we may obtain the desired bound directly by using $L^\infty \times L^2$ estimates with Lemma \ref{lem:coif-mey}.

\emph{Only one derivative falls on the phase function.} The derivative falling on the phase function implies the time growth $s$, we need to recover this growth. However, since the estimates for these cases coincide with proofs for first-order weighted energy estimates, we can finish the proof of these estimates similarly to those in Section \ref{sec:1st-dirac}.

\emph{All derivatives fall on the phase function.} As we mentioned, this case becomes our main case in our analysis. By Duhamel's formula \eqref{eq:duhamel-dirac}, we consider the following term: for $\thez,\theo,\thet \in \{+,-\}$, 
\begin{align}\label{eq:2nd-dirac-main}
	\int_0^t s^2 \int_{\R^4} e^{isp_\Theta(\xi,\eta)} [\nabla_\xi p_\Theta(\xi,\eta)]^2 \Pi_{\thez}(\xi) \wh{F_{\mu,\thet}}(s,\eta) \al^\mu \wh{f_{\theo}}(s,\xi-\eta) d\eta ds.
\end{align}

\subsection{Time resonance case: $\thez=\theo$} Since $\nabla_\xi p_\Theta(\xi,\eta)$ has a null structure when $\thez =\theo$, we exploit these null structures to obtain an extra time decay with the normal form approach. Integrating by parts in time, we see that \eqref{eq:2nd-dirac-main} can be bounded by the following terms:
	\begin{align}
		&t^2 \int_{\R^4} e^{itp_\Theta(\xi,\eta)} \frac{[\nabla_\xi p_\Theta(\xi,\eta)]^2}{p_\Theta(\xi,\eta)} \Pi_{\thez}(\xi) \wh{F_{\mu,\thet}}(t,\eta) \al^\mu \wh{f_{\theo}}(t,\xi-\eta) d\eta, \label{eq:2nd-dirac-1}\\
		&\int_0^t s \int_{\R^4} e^{isp_\Theta(\xi,\eta)} \frac{[\nabla_\xi p_\Theta(\xi,\eta)]^2}{p_\Theta(\xi,\eta)} \Pi_{\thez}(\xi) \wh{F_{\mu,\thet}}(s,\eta) \al^\mu \wh{f_{\theo}}(s,\xi-\eta) d\eta ds, \label{eq:2nd-dirac-2}\\
		&\int_0^t s^2 \int_{\R^4} e^{isp_\Theta(\xi,\eta)} \frac{[\nabla_\xi p_\Theta(\xi,\eta)]^2}{p_\Theta(\xi,\eta)} \Pi_{\thez}(\xi) \p_s\wh{F_{\mu,\thet}}(s,\eta) \al^\mu \wh{f_{\theo}}(s,\xi-\eta) d\eta ds, \label{eq:2nd-dirac-3}\\
		&\int_0^t s^2 \int_{\R^4} e^{isp_\Theta(\xi,\eta)} \frac{[\nabla_\xi p_\Theta(\xi,\eta)]^2}{p_\Theta(\xi,\eta)} \Pi_{\thez}(\xi) \wh{F_{\mu,\thet}}(s,\eta) \al^\mu \p_s\wh{f_{\theo}}(s,\xi-\eta) d\eta ds. \label{eq:2nd-dirac-4}
	\end{align}
Thanks to the null structure, the normal form transform gives an extra time decay without any singularity in \eqref{eq:2nd-dirac-1}--\eqref{eq:2nd-dirac-4}. Nevertheless, since the lack of decay property of spinor \eqref{eq:timedecay-d} still implies a $\bra{t}^{\beta+\de}$ divergence in terms of a priori assumption \eqref{eq:assumption-apriori}, we get an extra decay expect for estimates for \eqref{eq:2nd-dirac-1}--\eqref{eq:2nd-dirac-4}. 

\emph{1) Estimate for \eqref{eq:2nd-dirac-1}.} To bound \eqref{eq:2nd-dirac-1}, we decompose the frequencies $|\xi|,|\xi-\eta|,|\eta|$ into dyadic pieces $N_0,N_1,N_2 \in 2^\Z$, respectively and we denote  3-tuple $\textbf{N}=(N_0,N_1,N_2) \in 2^{\Z}\times 2^\Z\times 2^{\Z} $. Then one has
\begin{align}\label{eq:2nd-dirac-1-dec}
	&t^2 \int_{\R^4} e^{itp_\Theta(\xi,\eta)} \frac{[\nabla_\xi p_\Theta(\xi,\eta)]^2}{p_\Theta(\xi,\eta)}\rho_\textbf{N}(\xi,\eta) \Pi_{\thez}(\xi) \wh{F_{\mu,\thet,N_2}}(t,\eta) \al^\mu \wh{f_{\theo,N_1}}(t,\xi-\eta) d\eta,
\end{align}
where $\rho_\mathbf{N}(\xi,\eta) = \rho_{N_0}(\xi)\rho_{N_1}(\xi-\eta)\rho_{N_2}(\eta)$.
Let us first consider the high frequency regime $\bra{t}^{\frac3n} \le \bra{\nmax}$ by dividing into two parts: $\nmin \neq N_1$ and $\nmin = N_1$. By H\"older inequality, we estimate
\begin{align}
\begin{aligned}\label{eq:esti-n-min-1}
	&\sum_{\substack{\textbf{N},\bra{t}^{\frac3n}\le \bra{\nmax} \\\nmin \neq N_1}}\normo{ \eqref{eq:2nd-dirac-1-dec}}_{L_\xi^2} \\
	&\les \sum_{\substack{\textbf{N},\bra{t}^{\frac3n}\le \bra{\nmax} \\\nmin \neq N_1}} t^2  \normo{\frac{[\nabla_\xi p_\Theta(\xi,\eta)]^2}{p_\Theta(\xi,\eta)}\rho_\mathbf{N}(\xi,\eta)}_{L_{\xi,\eta}^\infty} \|\rho_{N_0}\|_{L_\xi^2} \left\|\wh{F_{\mu,\thet,N_2}}(t)\right\|_{L_\eta^2}\left\|\wh{f_{\theo,N_1}}(t)\right\|_{L_\eta^2}\\
	&\les \sum_{\substack{\textbf{N},\bra{t}^{\frac3n}\le \bra{\nmax} \\\nmin \neq N_1}} t^2  N_2 N_0^2 \bra{N_1}^{-n}\bra{N_2}^{-n} \|A_{\mu,\thet}(t)\|_{H^n}\|\psi_{\theo,N_1}(t)\|_{H^n}\\ 
	&\les \ve_1^2,
	\end{aligned}
\end{align}
where in the last inequality we used the fact that $\nmax \sim \max(N_1,N_2)$ from the frequency support relation and $\bra{N_1}^{-n}\bra{N_2}^{-n} \les \bra{t}^{-3}$. Changing of variables $\eta \mapsto \xi-\eta$ for $\nmin = N_1$, we see that
\begin{align}
	\begin{aligned}\label{eq:esti-n-min-2}
	&\sum_{\substack{\textbf{N},\bra{t}^{\frac3n}\le \bra{\nmax}   \\ \nmin = N_1}}\normo{ \eqref{eq:2nd-dirac-1-dec}}_{L_\xi^2} \\
	&\les\sum_{\substack{\textbf{N},\bra{t}^{\frac3n}\le \bra{\nmax}   \\ \nmin = N_1}} t^2  \normo{\frac{[\nabla_\xi p_\Theta(\xi,\xi-\eta)]^2}{p_\Theta(\xi,\xi-\eta)} \rho_{\mathbf N}}_{L_{\xi,\eta}^\infty}  \left\|\wh{F_{\mu,\thet,N_2}}(t)\right\|_{L_\xi^2} \|\rho_{N_1}\|_{L_\eta^2}\left\|\wh{f_{\theo,N_1}}(t)\right\|_{L_\eta^2}\\
	&\les \sum_{\substack{\textbf{N},\bra{t}^{\frac3n}\le \bra{\nmax}   \\ \nmin = N_1}} t^2  N_2 N_1^2 \bra{N_1}^{-n}\bra{N_2}^{-n} \|A_{\mu,\thet}(t)\|_{H^n}\|\psi_{\theo,N_1}(t)\|_{H^n}\\ 
	&\les \ve_1^2.
	\end{aligned}
\end{align}
Therefore, we will henceforth assume that $\bra{\nmax} \le \bra{t}^{\frac3n}$ in the estimates for \eqref{eq:2nd-dirac-1-dec}. To estimate \eqref{eq:2nd-dirac-1-dec}, the absence of time integral leads to a constraint to further use the normal form approach. We use the non-resonance  \eqref{eq:nonresonance-space} concerning frequency to overcome this absence of time integral, since \eqref{eq:2nd-dirac-1-dec} does not exhibit the space resonance regardless of the sign relations. Using the integration by parts in $\eta$ with relation
\begin{align*}
	e^{itp_\Theta(\xi,\eta)} = -i\frac{\nabla_\eta p_\Theta(\xi,\eta) \cdot  \nabla_\eta e^{itp_\Theta(\xi,\eta)}}{\left| \nabla_\eta p_\Theta(\xi,\eta)\right|^2},
\end{align*}
 one gets the following contributions: 
	\begin{align}
		& t \int_{\R^4} e^{itp_\Theta(\xi,\eta)} \nabla_\eta\mathbf{m_N(\xi,\eta)} \wh{F_{\mu,\thet,N_2}}(t,\eta) \al^\mu \wh{f_{\theo,N_1}}(t,\xi-\eta) d\eta, \label{eq:2nd-dirac-1-a}\\
		&t \int_{\R^4} e^{itp_\Theta(\xi,\eta)} \mathbf{m_N(\xi,\eta)} \wh{P_{N_2}(xF_{\mu,\thet})}(t,\eta) \al^\mu \wh{f_{\theo,N_1}}(t,\xi-\eta) d\eta, \label{eq:2nd-dirac-1-b}\\
		& t \int_{\R^4} e^{itp_\Theta(\xi,\eta)} \mathbf{m_N(\xi,\eta)} \wh{F_{\mu,\thet,N_2}}(t,\eta) \al^\mu \wh{P_{N_1}(xf_{\theo})}(t,\xi-\eta) d\eta, \label{eq:2nd-dirac-1-c}
	\end{align}
for $\textbf{N} \in 2^{3\Z}$, where the multiplier
\begin{align*}
 \mathbf{m_N(\xi,\eta)}:=	\frac{\nabla_\eta p_\Theta(\xi,\eta)[\nabla_\xi p_\Theta(\xi,\eta)]^2}{p_\Theta(\xi,\eta)|\nabla_\eta p_\Theta(\xi,\eta)|^2} \Pi_{\thez}(\xi)\rho_\textbf{N}(\xi,\eta).
\end{align*}
By simple calculation, we have 
\begin{align}\label{eq:multi-2nd-dirac-1}
	\normo{\frac{\mathbf{m_N}}{|\eta|(\bra{\xi-\eta} + \bra{\eta})^{20}}}_{\cm} \les 1 \;\;\mbox{ and }\;\; \normo{\frac{\nabla_\eta\mathbf{m_N}}{(\bra{\xi-\eta} + \bra{\eta})^{20}}}_{\cm} \les 1.
\end{align}
If $N_{012}^{\min} \le \bra{t}^{-1}$, in a  similar way to \eqref{eq:esti-n-min-1} and \eqref{eq:esti-n-min-2}, H\"older inequality yields that
\begin{align*}	\sum_{\substack{\textbf{N},\bra{\nmax} \le \bra{t}^{\frac3n}\\N_{012}^{\min} \le \bra{t}^{-1}}}\left\|\eqref{eq:2nd-dirac-1-a}\right\|_{L_\xi^2} \les \sum_{\substack{\textbf{N},\bra{\nmax} \le \bra{t}^{\frac3n}\\N_{012}^{\min} \le \bra{t}^{-1}}}t (N_{012}^{\min})^2 \bra{N_1}^{-n}\bra{N_2}^{-n} \normo{A_{\mu,\thet}(t)}_{H^n} \normo{\psi_\theo(t)}_{H^n} \les \ve_1^2,
\end{align*}
where we used the pointwise bound of $\nabla_\eta \mathbf{m_N}$. On the other hand, if $\bra{t}^{-1} \le N_{012}^{\min}  $, Lemma \ref{lem:coif-mey} with \eqref{eq:multi-2nd-dirac-1} implies that
\begin{align*}
	\sum_{\substack{\textbf{N},\bra{\nmax} \le \bra{t}^{\frac3n},\\ \bra{t}^{-1} \le N_{012}^{\min} }}\left\|\eqref{eq:2nd-dirac-1-a}\right\|_{L_\xi^2} \les t^{1+\zeta} \normo{A_{\mu,\thet}(t)}_{H^n}\normo{\psi_{\theo}(t)}_{W^{20,\infty}} \les \bra{t}^{-1+\zeta+\frac{3\beta}2+\frac{3\de}2}\ve_1^2 \les \ve_1^2,
\end{align*}
for some small $0<\zeta< 1-\frac{3\beta}2-\frac{3\de}2$, since   $\beta+\de < \frac14$. 

Let us move on to the estimates for \eqref{eq:2nd-dirac-1-b} and \eqref{eq:2nd-dirac-1-c}. Estimates on $N_{012}^{\min} \le \bra{t}^{-1}$ can be treated similarly to that of \eqref{eq:2nd-dirac-1-a}. Then Lemma \ref{lem:coif-mey}, Propositions \ref{prop:timedecay}, and \ref{prop:timedecay-maxwell} also give the bound
\begin{align*}
		\sum_{\substack{\textbf{N},\bra{\nmax} \le \bra{t}^{\frac3n},\\ \bra{t}^{-1} \le N_{012}^{\min} }}\left\|\eqref{eq:2nd-dirac-1-b}\right\|_{L_\xi^2} &\les 		\sum_{\substack{\textbf{N},\bra{\nmax} \le \bra{t}^{\frac3n},\\ \bra{t}^{-1} \le N_{012}^{\min} }}t \||D|xF_{\mu,\thet}(t)\|_{H^{20}}	 \normo{\psi_{\theo}(t)}_{W^{20,\infty}} \\
		&\les \bra t^{-1+\frac{3\beta}2+\frac{3\de}2 +\zeta}\ve_1^2
\end{align*}
and
\begin{align*}
		\sum_{\substack{\textbf{N},\bra{\nmax} \le \bra{t}^{\frac3n},\\ \bra{t}^{-1} \le N_{012}^{\min} }}	\left\|\eqref{eq:2nd-dirac-1-c}\right\|_{L_\xi^2}\les		\sum_{\substack{\textbf{N},\bra{\nmax} \le \bra{t}^{\frac3n},\\ \bra{t}^{-1} \le N_{012}^{\min} }} t \|A_{\mu,\thet}(t)\|_{W^{20,\infty}}	 \normo{xf_{\theo}(t)}_{H^{20}} \les \bra{t}^{-\frac12 +\de +\zeta}\ve_1^2.
\end{align*}
  In these estimates, $\zeta$ comes from the high frequency sum $\bra{\nmax} \le \bra{t}^\frac3n$ and satisfies that $0<\zeta<\min(1-\frac{3\beta}2-\frac{3\de}2,\frac12 -\de)$.

\emph{2) Estimate for \eqref{eq:2nd-dirac-2}.} For this case, we further perform the normal form transform. Then we have
	\begin{align}
			& t \int_{\R^4} e^{itp_\Theta(\xi,\eta)} \frac{[\nabla_\xi p_\Theta(\xi,\eta)]^2}{[p_\Theta(\xi,\eta)]^2} \Pi_{\thez}(\xi) \wh{F_{\mu,\thet}}(t,\eta) \al^\mu \wh{f_{\theo}}(t,\xi-\eta) d\eta ds, \label{eq:2nd-dirac-2-a}\\
			&\int_0^t  \int_{\R^4} e^{isp_\Theta(\xi,\eta)} \frac{[\nabla_\xi p_\Theta(\xi,\eta)]^2}{[p_\Theta(\xi,\eta)]^2} \Pi_{\thez}(\xi) \wh{F_{\mu,\thet}}(s,\eta) \al^\mu \wh{f_{\theo}}(s,\xi-\eta) d\eta ds, \label{eq:2nd-dirac-2-b}\\
			&\int_0^t s \int_{\R^4} e^{isp_\Theta(\xi,\eta)} \frac{[\nabla_\xi p_\Theta(\xi,\eta)]^2}{[p_\Theta(\xi,\eta)]^2} \Pi_{\thez}(\xi) \p_s \wh{F_{\mu,\thet}}(s,\eta) \al^\mu \wh{f_{\theo}}(s,\xi-\eta) d\eta ds, \label{eq:2nd-dirac-2-c}\\
			&\int_0^t s \int_{\R^4} e^{isp_\Theta(\xi,\eta)} \frac{[\nabla_\xi p_\Theta(\xi,\eta)]^2}{[p_\Theta(\xi,\eta)]^2} \Pi_{\thez}(\xi) \wh{F_{\mu,\thet}}(s,\eta) \al^\mu \p_s \wh{f_{\theo}}(s,\xi-\eta) d\eta ds, \label{eq:2nd-dirac-2-d}
	\end{align}
with the multiplier bound
\begin{align}\label{eq:multi-2nd-dirac-2}
\normo{	 \frac{[\nabla_\xi p_\Theta(\xi,\eta)]^2}{[p_\Theta(\xi,\eta)]^2 \left(\bra{\xi-\eta}+\bra{\eta}\right)^{20}}}_{\cm} \les 1.
\end{align}
Using Lemma \ref{lem:coif-mey} with \eqref{eq:multi-2nd-dirac-2}, we see that
\begin{align*}
	\left\|\eqref{eq:2nd-dirac-2-a}\right\|_{L_\xi^2} &\les t \|A_{\mu,\thet}(t)\|_{H^{n}}	 \normo{\psi_{\theo}(t)}_{W^{20,\infty}} \les  \ve_1^2,\\
	\left\|\eqref{eq:2nd-dirac-2-b}\right\|_{L_\xi^2} &\les \int_0^t \|A_{\mu,\thet}(s)\|_{H^{n}}	 \normo{\psi_{\theo}(s)}_{W^{20,\infty}} ds \les \ve_1^2.
\end{align*}
For estimates for \eqref{eq:2nd-dirac-2-c} and \eqref{eq:2nd-dirac-2-d}, Lemma \ref{lem:time-derivative} leads us that
\begin{align*}
	\left\|\eqref{eq:2nd-dirac-2-c}\right\|_{L_\xi^2} &\les \int_0^t \|\p_sF_{\mu,\thet}(s)\|_{H^{20}}	 \normo{\psi_{\theo}(s)}_{W^{20,\infty}} ds \les \int_0^t \bra{s}^{-3+\frac{3\beta}2+\frac{3\de}2}\ve_1^2 ds \les \ve_1^2,\\
	\left\|\eqref{eq:2nd-dirac-2-d}\right\|_{L_\xi^2} &\les \int_0^t \|A_{\mu,\thet}(s)\|_{W^{20,\infty}}	 \normo{\p_s f_{\theo}(s)}_{H^{20}} ds \les \int_0^t \bra{s}^{-\frac72+\beta+\de}\ve_1^2 ds \les \ve_1^2.
\end{align*}

\emph{3) Estimate for \eqref{eq:2nd-dirac-3} and \eqref{eq:2nd-dirac-4}.} As observed in the proof of Lemma \ref{lem:time-derivative}, one gets
\begin{align*}
	\p_s F_{\mu,\thet}(s) = -\thet \frac i2 \sum_{\kao,\kat \in \{\pm\}}  e^{-\thet it|D|}|D|^{-1}\bra{\psi_{\ka_1}(s),\alpha_\mu\psi_{\ka_2}(s)}.
\end{align*}
By H\"older inequality, we estimate
\begin{align*}
	\normo{|D|\p_s F_{\mu,\thet}(s)}_{H^{20}} \les \sum_{\kao,\kat \in \{\pm\}} \|\psi_{\kao}(s)\|_{H^{20}} \|\psi_{\kat}(s)\|_{W^{20,\infty}} \les \bra{s}^{-2+\beta+\de}\ve_1^2.
\end{align*}
Since
\begin{align*}
	\normo{	 \frac{[\nabla_\xi p_\Theta]^2}{p_\Theta|\eta| \left(\bra{\xi-\eta} + \bra{\eta}\right)^{20} } }_{\cm} \les 1,
\end{align*}
Lemma \ref{lem:coif-mey} leads us that
\begin{align*}
	\normo{\eqref{eq:2nd-dirac-3}}_{L_\xi^2} \les \int_0^t s^2 \||D|\p_s F_{\mu,\thet}(s)\|_{H^{20}} \|\psi_\theo(s)\|_{W^{20,\infty}} ds \les \int_0^t \bra{s}^{-2+2\beta+2\de} \ve_1^3 ds \les \ve_1^2,
\end{align*}
since $\ve_1<1$. To estimate \eqref{eq:2nd-dirac-4}, we obtain the multiplier estimates
\begin{align*}
	\normo{\frac{[\nabla_\xi p_\Theta]^2}{p_\Theta (\bra{\xi-\eta}+\bra{\eta})^{20}}}_{\cm} \les 1.
\end{align*}
By Proposition \ref{prop:timedecay-maxwell} and Lemma \ref{lem:time-derivative}, we estimate
\begin{align*}
	\normo{\eqref{eq:2nd-dirac-4}}_{L_\xi^2} \les \int_0^t s^2\normo{A_{\mu,\thet}(s)}_{W^{20,\infty}} \normo{e^{-\theo it\bra{D}}\p_s f_\theo(s)}_{H^{20}}  ds \les \int_0^t \bra{s}^{-\frac32 +\beta+\de} \ve_1^3 ds \les \ve_1^2,
\end{align*}

\subsection{Time non-resonance case: $\thez\neq\theo$} 
In contrast to the previous section, the multiplier $\nabla_\xi p_\Theta$ no longer possesses a null structure. Fortunately, as  examined in \eqref{eq:nonresonance-time}, we can use the normal form approach without any singularity by exploiting the time non-resonance of phase $p_\Theta$ when $\thez \neq \theo$. Indeed, by the integration by parts in time, \eqref{eq:2nd-dirac-main} can be bounded by the following terms:
\begin{align}	
	& t^2 \int_{\R^4} e^{itp_\Theta(\xi,\eta)} \textbf{M}(\xi,\eta) \wh{F_{\mu,\thet}}(t,\eta) \al^\mu \wh{f_{\theo}}(t,\xi-\eta) d\eta \label{eq:2-dirac-non-a}\\
	&  \int_0^t s \int_{\R^4} e^{isp_\Theta(\xi,\eta)} \textbf{M}(\xi,\eta) \wh{F_{\mu,\thet}}(s,\eta) \al^\mu \wh{f_{\theo}}(s,\xi-\eta) d\eta ds, \label{eq:2-dirac-non-b}\\
		&  \int_0^t s^2 \int_{\R^4} e^{isp_\Theta(\xi,\eta)} \textbf{M}(\xi,\eta) \partial_s\left(\wh{F_{\mu,\thet}}(s,\eta) \al^\mu \wh{f_{\theo}}(s,\xi-\eta) \right) d\eta ds ,\label{eq:2-dirac-non-c}
\end{align}
where the multiplier 
$$
\textbf{M}(\xi,\eta) := \frac{[\nabla_\xi p_\Theta(\xi,\eta)]^2 \Pi_{\thez}(\xi)}{p_\Theta(\xi,\eta)}.
$$
In particular, we note that $\normo{(\bra{\xi}+\bra{\eta})^{-20}\textbf{M}}_\cm \les 1$. We can further perform the normal form transform with the time non-resonance except for \eqref{eq:2-dirac-non-a}. This implies that the estimate of \eqref{eq:2-dirac-non-a} becomes the most complicated case in those of \eqref{eq:2-dirac-non-a}--\eqref{eq:2-dirac-non-c}.

We begin with the estimates for \eqref{eq:2-dirac-non-a}. In this case, the absence of time integral leads to a constraint to use the normal form approach. In view of \eqref{eq:nonresonance-space}, since \eqref{eq:2-dirac-non-a} does not exhibit the space resonance regardless of the sign relations, we proceed to bound \eqref{eq:2-dirac-non-a} with this non-resonance. Before that, we decompose $|\xi|,|\xi-\eta|,|\eta|$ in \eqref{eq:2-dirac-non-a} into the dyadic pieces $N_0,N_1,N_2 \in 2^\Z$, respectively. By doing this one gets that, for the 3-tuple $\textbf{N}=(N_0,N_1,N_2) \in 2^{\Z}\times 2^\Z\times 2^\Z $,
\begin{align}\label{eq:2-dirac-non-decom}
	 t^2 \int_{\R^4} e^{itp_\Theta(\xi,\eta)} \mathbf{M_N}(\xi,\eta) \wh{F_{\mu,\thet,N_2}}(t,\eta) \al^\mu \wh{f_{\theo,N_1}}(t,\xi-\eta) d\eta,
\end{align}
where $\mathbf{M_N}(\xi,\eta) := \mathbf{M}(\xi,\eta) \rho_{\textbf{N}}(\xi,\eta)$. Let us consider a low frequency regime. Analogously to estimates for \eqref{eq:2nd-dirac-1-dec}, using H\"older inequality and the change of variables, we can bound
\begin{align*}
	&\sum_{\textbf{N}, \nmax \ge \bra{t}^{\frac2n}}  \normo{\eqref{eq:2-dirac-non-decom}}_{L_\xi^2} \\
	&\hspace{1cm} \les \sum_{\textbf{N}, \nmax \ge \bra{t}^{\frac2n}} t^2\bra{N_0} \left( \|\rho_{N_0}\|_{L_\xi^2}+\|\rho_{N_1}\|_{L_\xi^2}\right) \bra{N_1}^{-n}\bra{N_2}^{-n} \normo{A_{\mu,\thet}(t)}_{H^n} \normo{\psi_\theo(t)}_{H^n}\\
	&\hspace{1cm} \les \bra{t}^{\frac\beta2 + \frac\de2}\ve_1^2 \les \bra{t}^{\de} \ve_1^2,
\end{align*}
since $\beta < \de$. Then we also have
\begin{align*}
	\sum_{\textbf{N}, \nmin \le \bra{t}^{-1}}  \normo{\eqref{eq:2-dirac-non-decom}}_{L_\xi^2} 	&\les \sum_{\textbf{N}, \nmin \le \bra{t}^{-1}} t^2 \bra{N_0}  (\nmin)^2 \bra{N_1}^{-n}\bra{N_2}^{-n} \normo{A_{\mu,\thet}(t)}_{H^n} \normo{\psi_\theo(t)}_{H^n}\\
	& \les \bra{t}^{\frac\beta2 + \frac\de2}\ve_1^2 \les \bra{t}^\de \ve_1^2.
\end{align*}
Therefore we henceforth assume that $\textbf{N} \in \mathcal N$, where 
$$
\mathcal N := \left\{ (N_0,N_1,N_2) : \nmax \le \bra{t}^{\frac2n} \;\mbox{ and }\;\nmin \ge \bra{t}^{-1}\right\}.
$$
As mentioned above, by integrating by parts in frequency, \eqref{eq:2-dirac-non-decom} can be bounded by the following:  
	\begin{align}	
		& t \int_{\R^4} e^{itp_\Theta(\xi,\eta)} \nabla_\eta \left(\frac{\nabla_\eta p_\Theta(\xi,\eta) \mathbf{M_N}(\xi,\eta)}{|\nabla_\eta p_\Theta(\xi,\eta)|^2} \right)\wh{F_{\mu,\thet,N_2}}(t,\eta) \al^\mu \wh{f_{\theo,N_1}}(t,\xi-\eta) d\eta, \label{eq:2-dirac-non-a-a}\\
		& t \int_{\R^4} e^{itp_\Theta(\xi,\eta)}  \frac{\nabla_\eta p_\Theta(\xi,\eta) \mathbf{M_N}(\xi,\eta)}{|\nabla_\eta p_\Theta(\xi,\eta)|^2}  \wh{P_{N_2}(xF_{\mu,\thet})}(t,\eta) \al^\mu \wh{f_{\theo,N_1}}(t,\xi-\eta) d\eta, \label{eq:2-dirac-non-a-b}\\
		& t \int_{\R^4} e^{itp_\Theta(\xi,\eta)}  \frac{\nabla_\eta p_\Theta(\xi,\eta) \mathbf{M_N}(\xi,\eta)}{|\nabla_\eta p_\Theta(\xi,\eta)|^2} \wh{F_{\mu,\thet,N_2}}(t,\eta) \al^\mu \wh{P_{N_1}(xf_{\theo})}(t,\xi-\eta) d\eta, \label{eq:2-dirac-non-a-c}
	\end{align}

To estimate \eqref{eq:2-dirac-non-a-a} and \eqref{eq:2-dirac-non-a-b}, we further take the integration by parts in frequency, obtaining that \eqref{eq:2-dirac-non-a-a} and \eqref{eq:2-dirac-non-a-b} are bounded by the following three terms, respectively:
\begin{align}\label{}
	&  \int_{\R^4} e^{itp_\Theta(\xi,\eta)} \nabla_\eta \mathbf{M_N^1}(\xi,\eta) \wh{F_{\mu,\thet,N_2}}(t,\eta) \al^\mu \wh{f_{\theo,N_1}}(t,\xi-\eta) d\eta, \label{eq:2-dirac-non-a-a-1}\\
	&  \int_{\R^4} e^{itp_\Theta(\xi,\eta)}  \mathbf{M_N^1}(\xi,\eta) \wh{P_{N_2}(xF_{\mu,\thet})}(t,\eta) \al^\mu \wh{f_{\theo,N_1}}(t,\xi-\eta) d\eta, \label{eq:2-dirac-non-a-a-2}\\
	&  \int_{\R^4} e^{itp_\Theta(\xi,\eta)}  \mathbf{M_N^1}(\xi,\eta) \wh{F_{\mu,\thet,N_2}}(t,\eta) \al^\mu \wh{P_{N_1}(xf_{\theo})}(t,\xi-\eta) d\eta, \label{eq:2-dirac-non-a-a-3}
\end{align}
and
\begin{align}
	&  \int_{\R^4} e^{itp_\Theta(\xi,\eta)} \nabla_\eta \mathbf{M_N^2}(\xi,\eta) \wh{P_{N_2}(xF_{\mu,\thet})}(t,\eta) \al^\mu \wh{f_{\theo,N_1}}(t,\xi-\eta) d\eta, \label{eq:2-dirac-non-a-b-1}\\
	&  \int_{\R^4} e^{itp_\Theta(\xi,\eta)}  \mathbf{M_N^2}(\xi,\eta) \wh{P_{N_2}(x^2F_{\mu,\thet})}(t,\eta) \al^\mu \wh{f_{\theo,N_1}}(t,\xi-\eta) d\eta, \label{eq:2-dirac-non-a-b-2}\\
	&  \int_{\R^4} e^{itp_\Theta(\xi,\eta)}  \mathbf{M_N^2}(\xi,\eta) \wh{P_{N_2}(xF_{\mu,\thet})}(t,\eta) \al^\mu \wh{P_{N_1}(xf_{\theo})}(t,\xi-\eta) d\eta, \label{eq:2-dirac-non-a-b-3}
\end{align}
where
\begin{align*}
	\mathbf{M_N^1}(\xi,\eta) := \frac{\nabla_\eta p_\Theta(\xi,\eta)}{|\nabla_\eta p_\Theta(\xi,\eta)|^2}\nabla_\eta \left(\frac{\nabla_\eta p_\Theta(\xi,\eta) \mathbf{m_N}(\xi,\eta)}{|\nabla_\eta p_\Theta(\xi,\eta)|^2}\right) \mbox{ and }   \mathbf{M_N^2}(\xi,\eta) := \frac{(\nabla_\eta p_\Theta(\xi,\eta))^2 \mathbf{m_N}(\xi,\eta)}{|\nabla_\eta p_\Theta(\xi,\eta)|^4}.
\end{align*}
Among the contribution of multipliers in \eqref{eq:2-dirac-non-a-a-1}--\eqref{eq:2-dirac-non-a-a-3}, we have
\begin{align}\label{eq:multi-2nd-dirac-a}
	\left\| \frac{ \mathbf{M_N^1}}{(\bra{\xi-\eta}+\bra{\eta})^{20}}\right\|_{\cm} \les N_2^{-1}  \;\;\mbox{ and } \;\; \left\| \frac{\nabla_\eta \mathbf{M_N^1}}{(\bra{\xi-\eta}+\bra{\eta})^{20}}\right\|_{\cm} \les N_2^{-2}.
\end{align}
Compared to $\nabla_\eta\mathbf{M_N^1}$, the multiplier $\mathbf{M_N^1}$ has a weaker singularity. Thus we handle \eqref{eq:2-dirac-non-a-a-1} carefully. As well as, since the weight of Maxwell part $F_{\mu,\thet}$ in \eqref{eq:2-dirac-non-a-a-2} gives rise to the time growth, the estimate of \eqref{eq:2-dirac-non-a-a-2} is also  complicated term.

By Lemma \ref{lem:coif-mey} with \eqref{eq:multi-2nd-dirac-a} and Bernstein's inequality, we estimate
\begin{align*}
	\sum_{\textbf{N} \in \mathcal N}\normo{\eqref{eq:2-dirac-non-a-a-1}}_{L_\xi^2} &\les \sum_{\textbf{N}\in \mathcal N} N_2^{-2} \bra{\nmax}^{20}\bra{N_1}^{-n} \normo{A_{\mu,\thet,N_2}(t)}_{L^\infty} \|\psi_{\theo}(t)\|_{H^n} \\
	&\les  \sum_{\textbf{N} \in \mathcal N}\bra{\nmax}^{20}\bra{N_1}^{-n} \normo{A_{\mu,\thet,N_2}(t)}_{L^2} \ve_1  \\
	&\les  \sum_{\textbf{N} \in \mathcal N}\bra{\nmax}^{20}\bra{N_1}^{-n}\bra{N_2}^{-n} \bra{t}^{\frac\beta2 +\frac\de2} \ve_1^2 \\
	& \les \bra{t}^{\frac\beta2 +\frac\de2+\zeta}\ve_1^2\les \bra{t}^{\de}\ve_1^2,
\end{align*}
 for any $\zeta$, which satisfies $0 <\zeta < \frac{\de-\beta}{2}$. Similarly, $L^\infty \times L^2$ estimates, Bernstein's inequality,  Sobolev embedding $\doth^1 \hookrightarrow L^4$, and Lemma \ref{lem:1st-weight-high} yield that 
\begin{align*}
	\sum_{\textbf{N} \in \mathcal N}\normo{\eqref{eq:2-dirac-non-a-a-2}}_{L_\xi^2} &\les \sum_{\textbf{N}\in \mathcal N} N_2^{-1} \bra{\nmax}^{20}\bra{N_1}^{-n} \normo{P_{N_2}(e^{\thet it|D|}xF_{\mu,\thet})(t)}_{L^\infty} \|\psi_{\theo}(t)\|_{H^n} \\
	&\les  \sum_{\textbf{N} \in \mathcal N}\bra{\nmax}^{20}\bra{N_1}^{-n}\bra{N_2}^{-20} \normo{|D|xF_{\mu,\thet}}_{H^{20}} \ve_1 \les \bra{t}^{ \frac\beta2+\frac\de2+\zeta}\ve_1^2 \le \bra{t}^\de \ve_1^2,
\end{align*}
 for any  $0 <\zeta < \frac{\de-\beta}{2}$. The estimate of \eqref{eq:2-dirac-non-a-a-3} can be obtained more simply.

Estimates of \eqref{eq:2-dirac-non-a-b-1}--\eqref{eq:2-dirac-non-a-b-3} can be finished by using Bernstein's inequality and $L^\infty \times L^2$ estimates with the multiplier bounds
\begin{align*}
	\left\| \frac{ \mathbf{M_N^2}}{(\bra{\xi-\eta}+\bra{\eta})^{20}}\right\|_{\cm} \les 1  \;\;\mbox{ and } \;\; \left\| \frac{\nabla_\eta \mathbf{M_N^2}}{(\bra{\xi-\eta}+\bra{\eta})^{20}}\right\|_{\cm} \les N_2^{-1}.
\end{align*}

We finish this section with the estimates of \eqref{eq:2-dirac-non-a-c}. we obtain, by $L^\infty \times L^2$ estimates, that
\begin{align*}
	\sum_{\textbf{N}\in \mathcal N}\normo{ \eqref{eq:2-dirac-non-a-c}}_{L_\xi^2} \les \sum_{\textbf{N}\in \mathcal N} t \bra{\nmax}^{20}\bra{N_1}^{-20}\bra{N_2}^{-20} \|A_{\mu,\thet,N_2}(t)\|_{W^{20, \infty}}\|xf_\theo(t)\|_{H^{20}}  \les\ve_1^2,
\end{align*}
for $0 < \zeta  \ll 1$ with the high regularity condition.

It remains to consider \eqref{eq:2-dirac-non-b} and \eqref{eq:2-dirac-non-c}. Since the estimates for \eqref{eq:2-dirac-non-c} can be done similarly to those for \eqref{eq:2-dirac-non-b} with Lemma \ref{lem:time-derivative}, we only treat  \eqref{eq:2-dirac-non-b}. We use once more the integration by parts in time and this implies that \eqref{eq:2-dirac-non-b} is given by the following three terms:
	\begin{align}	
		& t \int_{\R^4} e^{itp_\Theta(\xi,\eta)} \frac{\textbf{m}(\xi,\eta)}{p_\Theta(\xi,\eta)} \wh{F_{\mu,\thet}}(t,\eta) \al^\mu \wh{f_{\theo}}(t,\xi-\eta) d\eta, \label{eq:2-time-d-b-a}\\
		&  \int_0^t  \int_{\R^4} e^{isp_\Theta(\xi,\eta)} \frac{\textbf{m}(\xi,\eta)}{p_\Theta(\xi,\eta)} \wh{F_{\mu,\thet}}(s,\eta) \al^\mu \wh{f_{\theo}}(s,\xi-\eta) d\eta ds,\label{eq:2-time-d-b-b}\\
		&  \int_0^t s \int_{\R^4} e^{isp_\Theta(\xi,\eta)} \frac{\textbf{m}(\xi,\eta)}{p_\Theta(\xi,\eta)} \partial_s \left[\wh{F_{\mu,\thet}}(s,\eta) \al^\mu \wh{f_{\theo}}(s,\xi-\eta) \right] d\eta ds,\label{eq:2-time-d-b-c}		
	\end{align}
Then, using \eqref{eq:timedecay-m} and integration by parts in time with $\normo{\frac{\textbf{m}}{p_\Theta(\bra{\xi}+\bra{\eta})^{20}}}_\cm \les 1$, Propositions \ref{prop:timedecay}, \ref{prop:timedecay-maxwell}, and Lemma \ref{lem:coif-mey} lead us that
\[
\left\| \eqref{eq:2-time-d-b-a} \right\|_{L_\xi^2} \les  t \left( \|A_{\mu,\thet}(t)\|_{H^{n}} \|\psi(t)\|_{W^{20,\infty}} + \|A_{\mu,\thet}(t)\|_{W^{20,\infty}} \|\psi(t)\|_{H^{n}}\right)   \les \ve_1^2.
\]
The estimate \eqref{eq:2-time-d-b-b} is the analogue of \eqref{eq:2-time-d-b-a} with the time integral instead of time growth $t$. Estimates for \eqref{eq:2-time-d-b-c} can be obtained in a straightforward fashion by Lemma \ref{lem:time-derivative}. We omit the details.

\section{First order weighted energy estimates for gauge: Proof of \eqref{eq:1st-moment-m}}

By Plancherel's theorem, one gets that
\[
\normo{x F_{\mu,\ka}(t)}_{\doth^1} = \normo{ |\xi|\nabla_\xi \wh{F_{\mu,\ka}}(\xi)}_{L_\xi^{2}}.
\]
Recalling the notation \eqref{eq:interaction-maxwell}, \eqref{eq:phase-maxwell}, and Duhamel's formula \eqref{eq:duhamel-maxwell}, $|\xi|\nabla_\xi F_{\mu,\kaz}$ can be represented by the following:
\begin{align}
& \sum_{\ka_1,\kat \in \{\pm\}}\int_{0}^{t}\!\!\int_{\R^4}  e^{is q_K(\xi,\eta)} \frac{\xi}{|\xi|^2} \bra{\widehat{f_\kat}(s,\eta),\alpha_\mu\wh{f_\kao}(s,\xi+\eta)}\, d \eta ds \label{eq:1-maxwell-a}\\
	& \sum_{\ka_1,\kat \in \{\pm\}} \int_{0}^{t}\!\!\int_{\R^4}  e^{is q_K(\xi,\eta)}  \bra{\widehat{f_\kat}(s,\eta),\alpha_\mu \nabla_\xi \wh{f_\kao}(s,\xi+\eta)}\, d \eta ds,\label{eq:1-maxwell-b}\\
	& \sum_{\ka_1,\kat \in \{\pm\}}\int_{0}^{t}s\int_{\R^4}  e^{is q_K(\xi,\eta)} \nabla_{\xi} q_K(\xi,\eta) \bra{\widehat{f_\kat}(s,\eta),\alpha_\mu\wh{f_\kao}(s,\xi+\eta)}\, d \eta ds, \label{eq:1-maxwell-c}
\end{align}
where $q_K$ is defined in \eqref{eq:phase-maxwell}.

For the estimates for \eqref{eq:1-maxwell-a} and \eqref{eq:1-maxwell-b}, a priori assumption \eqref{eq:assumption-apriori}, Lemma \ref{lem:hls}, and \eqref{eq:timedecay-d}  deduce that
\begin{align*}
	\|\eqref{eq:1-maxwell-a}\|_{L_\xi^2} \les \int_0^t   \normo{\bra{\psi_{\kat}(s), \al^\mu \psi_{\kao}(s)}}_{L^\frac43} ds \les \int_0^t \bra{s}^{-1+\frac\beta2 +\frac\de2} \ve_1^2 ds \les  \bra{t}^{\frac\beta2 +\frac\de2} \ve_1^2
\end{align*}
and
\begin{align*}
	\|\eqref{eq:1-maxwell-b}\|_{L_\xi^2} \les \int_0^t   \normo{\psi_\kat(s)}_{L^\infty} \normo{  xf_{\kao}(s)}_{L^2} ds \les \int_0^t \bra{s}^{-2+\beta+\de} \ve_1^2 ds \les  \ve_1^2.
\end{align*}

Let us consider \eqref{eq:1-maxwell-c}. In this case, the multiplier $\nabla_\xi q_K$ does not give any null effect even though we consider all sign relations. The direct $L^2 \times L^\infty$ estimate gives rise to the $\bra{t}^{\beta+\de}$ growth, this estimate is failure given the a priori assumption \eqref{eq:assumption-apriori}. To handle this divergence, we use the time resonance of $q_K$ in \eqref{eq:resonance-time-maxwell}. Integrating by parts in time, we see that \eqref{eq:1-maxwell-c} can be bounded by
\begin{align}
&t\int_{\R^4} e^{it q_K(\xi,\eta)} \frac{\nabla_{\xi} q_K(\xi,\eta)}{q_K(\xi,\eta)}  \bra{\widehat{f_\kat}(t,\eta),\alpha_\mu\wh{f_\kao}(t,\xi+\eta)}\, d \eta,	\label{eq:1-maxwell-c-1}\\
&\int_{0}^{t}\int_{\R^4} e^{is q_K(\xi,\eta)} \frac{\nabla_{\xi} q_K(\xi,\eta)}{q_K(\xi,\eta)}   \bra{\widehat{f_\kat}(s,\eta),\alpha_\mu\wh{f_\kao}(s,\xi+\eta)}\, d \eta ds,\label{eq:1-maxwell-c-2}\\
&\int_{0}^{t}s\int_{\R^4} e^{is q_K(\xi,\eta)}  \frac{\nabla_{\xi} q_K(\xi,\eta)}{q_K(\xi,\eta)}  \p_s\bra{\widehat{f_\kat}(s,\eta),\alpha_\mu\wh{f_\kao}(s,\xi+\eta)}\, d \eta ds,\label{eq:1-maxwell-c-3}
\end{align}
for $\kao,\kat \in \{+,-\}$.  Since  \eqref{eq:1-maxwell-c-2}  can be  estimated  analogously to those of \eqref{eq:1-maxwell-c-2}, we only treat \eqref{eq:1-maxwell-c-1} and \eqref{eq:1-maxwell-c-2}

\emph{Estimates for \eqref{eq:1-maxwell-c-1}.} Let us first consider \eqref{eq:1-maxwell-c-1}. To estimate \eqref{eq:1-maxwell-c-1}, we proceed by dividing the sign relation into time resonance case $\kao = \kat$
 and  time non-resonance case$\kao \neq \kat$.

\textbf{The time resonance case: $\kao =\kat$.} Let us  decompose $|\xi|,|\xi+\eta|,|\eta|$ in \eqref{eq:1-maxwell-c-1} into dyadic numbers $N_0,N_1,N_2 \in 2^\Z$, respectively. Then we obtain that 
\begin{align}\label{eq:1-maxwell-c-1-de}
	&t\int_{\R^4} e^{it q_K(\xi,\eta)} \frac{\nabla_{\xi} q_K(\xi,\eta)}{q_K(\xi,\eta)} \rho_\textbf{N}(\xi,\eta) \bra{\widehat{f_{\kat,N_2}}(t,\eta),\alpha_\mu\wh{f_{\kao,N_1}}(t,\xi+\eta)}\, d \eta,	
\end{align}
where $\rho_\textbf{N}(\xi,\eta)= \rho_{N_0}(
\xi)\rho_{N_1}(\xi+\eta)\rho_{N_2}(\eta)$. By \eqref{eq:resonance-time-maxwell}, we see that
\begin{align}\label{eq:1-maxwell-multi}
	\normo{\frac{\nabla_{\xi} q_K \rho_\textbf{N}}{q_K}}_\cm \les \bra{\nmax}^{20}N_0^{-1}.
\end{align}
To bound \eqref{eq:1-maxwell-c-1-de}, we decompose the case into the following two cases (i) $\nmin \le \bra{t}^{-1+\theta}$ and (ii) $ \nmin \ge \bra{t}^{-1+\theta}$, where  $\theta = \frac\beta2 + \frac\de 2$. The estimates for Case (i) can be obtained,  by using   H\"older inequality and Sobolev embedding, as follows:
\begin{align*}
	\sum_{\textbf{N}, \nmin \le \bra{t}^{-1+\theta}} \left\|\eqref{eq:1-maxwell-c-1-de}\right\|_{L_\xi^2} &\les t \sum_{\textbf{N}, \nmin \le \bra{t}^{-1+\theta}} N_0^{-1}\|\rho_{N_0}\|_{L_\xi^2} N_1N_2 \normo{\wh{f_{\kat,N_2}}(t)}_{L_\eta^4} \normo{\wh{f_{\kao,N_1}}(t)}_{L_\eta^4} \\
	&\les t \sum_{\textbf{N}, \nmin \le \bra{t}^{-1+\theta}} N_0N_1N_2 \normo{P_{N_2}(xf_{\kat})(t)}_{L^2} \normo{P_{N_1}(xf_{\kao})(t)}_{L^2} \\
	&\les t \ve_1^2\sum_{\textbf{N}, \nmin \le \bra{t}^{-1+\theta}} N_0N_1 N_2 \bra{N_1}^{-20}\bra{N_2}^{-20}\\
	&\les \bra{t}^{\theta} \ve_1^2.
\end{align*}

For the Case (ii), $L^2 \times L^\infty$ estimates give rise to divergence, which comes from frequency sum. To overcome this divergence, We further exploit the space resonance \eqref{eq:resonance-space-maxwell} by using the integration by parts in $\eta$. Then 
\begin{align}
&\int_{\R^4} e^{it q_K(\xi,\eta)} \nabla_\eta\mathbf{L_N}(\xi,\eta) \bra{\widehat{f_{\kat,N_2}}(t,\eta),\alpha_\mu\wh{f_{\kao,N_1}}(t,\xi+\eta)}\, d \eta,	\label{eq:1-maxwell-c-1-1}\\
&\int_{\R^4} e^{it q_K(\xi,\eta)} \mathbf{L_N}(\xi,\eta) \bra{\nabla_\eta\widehat{f_{\kat,N_2}}(t,\eta),\alpha_\mu\wh{f_{\kao,N_1}}(t,\xi+\eta)}\, d \eta,	\label{eq:1-maxwell-c-1-2}	
\end{align} 
and a symmetric term, where
\begin{align*}
	\mathbf{L_N}(\xi,\eta) := \frac{\nabla_{\xi} q_K(\xi,\eta)\nabla_\eta q_K(\xi,\eta)}{q_K(\xi,\eta)|\nabla_\eta q_K(\xi,\eta)|^2} \rho_\textbf{N}(\xi,\eta).
\end{align*}
For the multiplier bounds, simple calculation leads us that
\begin{align*}
	\left\|\nabla_\eta\mathbf{L_N}\right\|_{L_{\xi,\eta}^\infty} \les N_0^{-2}N_2^{-1}\bra{\nmax}^{7} \;\;\mbox{ and }\;\; \left\|\mathbf{L_N}\right\|_{L_{\xi,\eta}^\infty} \les N_0^{-2}\bra{\nmax}^5.
\end{align*}
Using these pointwise bound of multiplier and H\"older inequality, we estimate that 
\begin{align*}
	&\sum_{\textbf{N}, \nmin \ge \bra{t}^{-1+\theta}} \left\|\eqref{eq:1-maxwell-c-1-1}\right\|_{L_\xi^2}\\
	 & \les \sum_{\textbf{N}, \nmin \ge \bra{t}^{-1+\theta}}\!\! \bra{\nmax}^{7}N_0^{-2}N_2^{-1}  \normo{\rho_{N_0}}_{L_\xi^2} \normo{\rho_{N_1}\rho_{N_2}}_{L_\eta^2} \normo{\wh{f_{\kat,N_2}}(t)}_{L_\eta^4}\normo{\wh{f_{\kat,N_1}}(t)}_{L_\eta^4}\\
	&\les \sum_{\textbf{N}, \nmin \ge \bra{t}^{-1+\theta}}\!\! \bra{\nmax}^{-18} N_1^{\frac12} N_2^\frac12  \prod_{j=1}^{2}\normo{xf_{\ka_j,N_j}(t)}_{H^{20}}\\
	&\les  \bra{t}^{\zeta}\ve_1^2,
\end{align*}
for arbitrary  small $0 <\zeta \ll 1$. Similarly, \eqref{eq:1-maxwell-c-1-2} can be estimated.

\textbf{The time non-resonance case: $\kao \neq \kat$.} Since this case does not possess time resonance, we may estimate more simply to the resonance case. Indeed, compared to \eqref{eq:1-maxwell-multi}, we have better bound 
\begin{align*}
	\normo{\frac{\nabla_\xi q_K}{q_K(\bra{\xi}+\bra{\eta})^{20}}}_{\cm} \les 1,
\end{align*}
which implies that
\begin{align*}
	\|\eqref{eq:1-maxwell-c-1}\|_{L_\xi^2} \les t \normo{\psi_\kao(t)}_{H^n}\normo{\psi_\kat(t)}_{W^{20,\infty}} \les \ve_1^2.
\end{align*}

\emph{Estimates for \eqref{eq:1-maxwell-c-3}.}  As observed in the estimates for \eqref{eq:1-maxwell-c-1}, it suffices to consider only the time resonance case $\kao =\kat$. Let us first decompose only $|\xi|$ into a dyadic number $N_0 \in 2^\Z$.

For \eqref{eq:1-maxwell-c-3},   with the multiplier bound
\begin{align*}
	\normo{\frac{\nabla_\xi q_K}{q_K} \rho_{N_0}(\xi)}_{L_{\xi,\eta}^\infty} \les N_0^{-1} \;\;\mbox{ and }\;\; \normo{\frac{\nabla_\xi q_K}{q_K(\bra{\xi}+\bra{\eta})^{20}} \rho_{N_0}(\xi)}_{\cm} \les N_0^{-1},
\end{align*} 
Using  H\"older inequality for low frequency regime and Lemma \ref{lem:coif-mey} for high frequency regime, we see that
\begin{align*}
	\|\eqref{eq:1-maxwell-c-3}\|_{L_\xi^2} &\les \int_0^t \sum_{N_0 \le \bra{s}^{-\frac12}}s N_0^{-1} \|\rho_{N_0}\|_{L_\xi^2} \|e^{-\kat is \bra{D}}\p_s f_\kat(s)\|_{H^{20}} \|\psi_\kao(s)\|_{L^2}  ds \\
	&\hspace{2cm} +\int_0^t \sum_{N_0 \ge \bra{s}^{-\frac12}}s N_0^{-1}  \|e^{-\kat is \bra{D}}\p_s f_\kat(s)\|_{H^{20}} \|\psi_\kao(s)\|_{W^{20,\infty}}  ds \\
	&\les \int_0^t  \bra{s}^{-\frac32+\frac{3\beta}2+\frac{3\de}2} \ve_1^2ds \les \ve_1^2,
\end{align*}
where we used Lemma \ref{lem:time-derivative}. This finishes the proof of \eqref{eq:1st-moment-m}.

\section{Second order weighted estimates for gauge: Proof of \eqref{eq:2nd-moment-m}}\label{sec:2nd-esti-maxwell}
We begin with
\[
\normo{x^2 F_{\mu,\ka}(t)}_{\doth^2} = \normo{|\xi|^2 \nabla_\xi^2 \wh{F_{\mu,\ka}}(t,\xi)}_{L_\xi^{2}}.
\]
Similarly to the above section, two derivatives $\nabla_\xi^2$ fall on ether $|\xi|^{-1}$, $e^{is q_K(\xi,\eta)}$, and $\wh{f_\kao}(\xi+\eta)$. Then we obtain six cases and we brief for the estimates of each cases.

\emph{$\nabla_\xi^2$ falls on all $\wh{f_\kao}$.} This case can be estimated by H\"older inequality and Proposition \ref{prop:timedecay}.

\emph{$\nabla_\xi^2$ falls on all $|\xi|^{-1}$ and $\nabla_\xi^2$ fall on $|\xi|^{-1}$ and $\wh{f_\kao}(\xi+\eta)$, respectively.} These two cases can be estimated by Hardy-Littlewood-Sobolev inequality.

\emph{$\nabla_\xi$ falls on the phase function $e^{isq_K(\xi,\eta)}$ and another $\nabla_\xi$ falls on $\nabla_\xi q_K$, $|\xi|^{-1}$,  or $\wh{f_\kao}(\xi+\eta)$.} By a $L^\infty \times L^2$ estimates using Lemma \ref{lem:coif-mey}, we readily obtain the desired bounds for these two cases.

\emph{$\nabla_\xi^2$ falls on the phase function $e^{isq_K(\xi,\eta)}$.} The most complicated term appears in this case:
\begin{align}\label{eq:2nd-maxwell}
|\xi|\int_{0}^{t}s^2 \!\! \int_{\R^4} e^{is q_K(\xi,\eta)}  \left[\nabla_{\xi} q_K(\xi,\eta) \right]^2  \bra{\widehat{f_\kat}(s,\eta),\alpha_\mu\wh{f_\kao}(s,\xi+\eta)}\, d \eta ds, 
\end{align}
for $\kao, \kat \in \{+,-\}$. In this case, since we need to recover the time growth, while $[\nabla_\xi q_K]^2$ does not possess a null structure that can eliminate the resonances. In view of \eqref{eq:resonance-space-maxwell}, $\kao=\kat$ implies the space resonance. We omit the estimates for non-resonant case $\kao \neq \kat$. Using integration by parts in time, \eqref{eq:2nd-maxwell} can be bounded by the following terms:
Similarly to above section, we only consider the resonant case $\kao = \kat$. Similarly to the previous section, by the integration by parts in time, we obtain the following types of contributions:
	\begin{align}	
		&t^2 \int_{\R^4}  e^{it q_K(\xi,\eta)}  \mathbf{L}(\xi,\eta) \bra{\widehat{f_\kat}(t,\eta),\alpha_\mu\wh{f_\kao}(t,\xi+\eta)}\, d \eta , \label{eq:2nd-maxwell-time-a}\\
		&\int_0^t s \int_{\R^4}  e^{it q_K(\xi,\eta)}  \mathbf{L}(\xi,\eta) \bra{x\widehat{f_\kat}(s,\eta),\alpha_\mu\wh{f_\kao}(s,\xi+\eta)}\, d \eta ds, \label{eq:2nd-maxwell-time-b}\\
		&  \int_0^t  s^2 \int_{\R^4}  e^{is q_K(\xi,\eta)}  \mathbf{L}(\xi,\eta) \p_s \bra{\widehat{f_\kat}(s,\eta),\alpha_\mu\wh{f_\kao}(s,\xi+\eta)}\, d \eta ds, \label{eq:2nd-maxwell-time-c}
	\end{align}
 where the multiplier 
$$
 \mathbf{L}(\xi,\eta) := \frac{|\xi| \left[\nabla_{\xi} q_K(\xi,\eta) \right]^2 }{ q_K(\xi,\eta)}.
$$
Then we have the multiplier bound
\begin{align*}
	\normo{\frac{\textbf{L}}{(\bra{\xi}+\bra{\eta})^{20}}}_{\cm} \les 1.
\end{align*}
From this estimate, we use Lemma \ref{lem:coif-mey} to deduce that
\begin{align*}
\normo{\eqref{eq:2nd-maxwell-time-a}}_{L_\xi^2} \les t^2 \|\psi_{\kao,N_1}(t)\|_{W^{20,\infty}}\|\psi_{\kat,N_2}(t)\|_{H^n} ds  \les \bra{t}^{\beta+\de}\ve_1^2 \les \bra{t}^\frac18 \ve_1^2.
\end{align*}
The estimate of \eqref{eq:2nd-maxwell-time-b} can be treated in the same way. To conclude the bound of \eqref{eq:2nd-maxwell-time-c}, we use Lemma \ref{lem:time-derivative}, obtaining
\begin{align*}
	\|\eqref{eq:2nd-maxwell-time-c}\|_{L_\xi^2} &\les \int_0^t s^2 \left( \|\p_s f_{\kao,N_1}(s)\|_{H^{20}}\| \psi_{\kat,N_2}(s)\|_{W^{20,\infty}}+ \| \psi_{\kao,N_1}(s)\|_{W^{20,\infty}}\|\p_s f_{\kat,N_2}(s)\|_{H^{20}} \right) ds \\
	&\les \int_0^t\bra{s}^{-2+\frac{5\beta}2+\frac{5\de}2} \ve_1^2ds \les \ve_1^2. 
\end{align*}
Therefore, we conclude \eqref{eq:2nd-moment-m} from the estimates \eqref{eq:2nd-maxwell-time-a}--\eqref{eq:2nd-maxwell-time-c}.

\section*{Acknowledgement}
The author was supported
in part by NRF-2022R1I1A1A0105640812 and  NRF-2019R1A5A1028324,
the National Research Foundation of Korea(NRF) grant funded by the Korea government (MOE) and (MSIT), respectively.

\bibliographystyle{plain}
\bibliography{ReferencesKiyeon}

\medskip

\end{document}